\documentclass{amsart} 
\usepackage[left=1.3in, right=1.2in, top=1in, bottom=1in, includefoot]{geometry}

\usepackage{pgfplots}
\pgfplotsset{compat=1.16}
\usepackage{adjustbox}

\usepackage{soul,xcolor}
\usepackage{bbm}

\usepackage{dsfont}

\usepackage{amsmath}
\usepackage{amsfonts}
\usepackage{amssymb}

\usepackage{amssymb}
\usepackage{enumitem}

\usepackage[nocompress]{cite}

\usepackage{hyperref}

\hypersetup{
	colorlinks   = true, 
	urlcolor     = blue, 
	linkcolor    = blue, 
	citecolor   = red 
}

\usepackage{caption}
\usepackage{todonotes}

\newtheorem{thm}{Theorem}[section]
\newtheorem{prop}[thm]{Proposition}
\newtheorem{lem}[thm]{Lemma}
\newtheorem{cor}[thm]{Corollary}

\theoremstyle{definition}
\newtheorem{definition}[thm]{Definition}

\DeclareRobustCommand{\gobblefive}[5]{}
\newcommand*{\SkipTocEntry}{\addtocontents{toc}{\gobblefive}}

\theoremstyle{remark}

\newtheorem{remark}[thm]{Remark}

\numberwithin{equation}{section}

\newcommand{\R}{\mathbb{R}}  
\newcommand{\N}{\mathbb{N}}  

\usepackage{float}
\usepackage{amsmath}
\usepackage{amssymb} 
\usepackage{mathtools}
\usepackage{mathabx}

\newcommand\norm[1]{\lVert#1\rVert}
\renewcommand\d{\:\mathrm{d}}

\newcommand{\scalarprod}[2]{\big({#1},{#2}\big)}
\newcommand{\jbrack}[1]{\langle{#1}\rangle}

\newcommand{\Abs}[1]{\left\vert #1 \right\vert}
\newcommand{\abs}[1]{\vert #1 \vert}

\newcounter{rtaskno}

\newcounter{rsubtaskno}

\usepackage{scrextend}

\usepackage{xcolor}
\definecolor{green}{rgb}{0.01, 0.75, 0.24}
\definecolor{aqua}{rgb}{0.0, 0.44, 1.0}

\usepackage[scr=boondox]  
           {mathalpha}

\begin{document}

	\title[The Whitham equation and a Whitham-Boussinesq system]{Improved existence time for the Whitham equation and a Whitham-Boussinesq system} 


\author[D. Pilod]{Didier Pilod}
\address{Department of Mathematics\\ University of Bergen\\ Postbox 7800\\ 5020 Bergen\\ Norway}
\email{Didier.Pilod@uib.no}

\author[S. Selberg]{Sigmund Selberg}
\address{Department of Mathematics\\ University of Bergen\\ Postbox 7800\\ 5020 Bergen\\ Norway}
\email{Sigmund.Selberg@uib.no}

\author[N. S. Taki]{Nadia Skoglund Taki}
\address{Department of Mathematics\\ University of Bergen\\ Postbox 7800\\ 5020 Bergen\\ Norway}
\email{Nadia.Taki@uib.no}


\author[A. Tesfahun]{Achenef Tesfahun}
\address{Department of Mathematics\\ Nazarbayev University\\ Qabanbai Batyr Avenue 53\\ 010000 Nur-Sultan\\ Republic of Kazakhstan}
\email{Achenef.Tesfahun@nu.edu.kz}



\date{\today}

\begin{abstract}
In this paper, we investigate the time of existence of the solutions to two full dispersion models derived from the water waves equations in the shallow water regime: the Whitham equation and a Whitham-Boussinesq system in dimension one and two. The regime is characterized by the nonlinearity parameter $\epsilon\in(0,1]$ and the shallow water parameter $\mu\in(0,1]$. 

We extend the lifespan of the solution beyond the hyperbolic time $\epsilon^{-1}$. More precisely, we establish well-posedness on the timescale of order $\mu^{\frac{1}{4}^-}\epsilon^{(-\frac{5}{4})^+}$ in the one-dimensional case, and of order $\mu^{\frac{1}{4}^-}\epsilon^{(-\frac{3}{2})^+}$ in dimension two. We emphasize that for the two-dimensional case, we obtain a time of existence of order $\epsilon^{-\frac{5}{4}}$ in the long wave regime $\mu \sim \epsilon$. This kind of result seems to be new, even for the Boussinesq systems. 

The proofs combine energy methods with Strichartz estimates. Here, a key ingredient is to obtain new refined Strichartz estimates that include the small parameter $\mu$. These techniques are robust and could be adapted to improve the lifespan of solutions for other equations and systems of the same form.
\end{abstract}


\maketitle

\section{Introduction}\label{sec:intro}

\subsection{General setting}
In general, nonlinear dispersive equations and systems are not derived from first principles. Instead, they arise as asymptotic models obtained from more general systems under a suitable scaling as small parameters tend to zero. These asymptotic models are intended to capture the dynamics of the original system within the regime defined by these scaling assumptions. In this article, we focus on general models of the form 
\begin{equation} \label{general:model}
\partial_tU+L_\mu U+\epsilon F(U,\nabla U)=0 ,
\end{equation}
where  $\epsilon$ and $\mu$ are small parameters taking into account the nonlinear and dispersive effects, respectively. Here, $L_\mu$ is a linear dispersive operator and $F(U,\nabla U)$ is a quadratic nonlinearity. 

By using hyperbolic techniques, the time of existence of such systems is typically of order $1/\epsilon$.  The main objective of this paper is to improve the lifespan with respect to the small parameters $\epsilon$ and $\mu$ by exploiting the dispersive effects of the model. 

A typical example of \eqref{general:model} is the {\it abcd} Boussinesq systems modelling surface water waves given by 
 \begin{equation}
    \label{abcd}
    \left\lbrace
    \begin{array}{l}
    \eta_t+\nabla \cdot {\bf v}+\epsilon \nabla\cdot(\eta {\bf v})+ \mu [a\nabla\cdot \Delta{\bf v}-b  \Delta \eta_t]=0 \\
    {\bf v}_t+\nabla \eta+ \frac{\epsilon}{2}\nabla (|{\bf v}|^2)+\mu [c\nabla \Delta \eta-d \Delta {\bf v}_t]=0.
    \end{array}\right.
    \end{equation}
Here, $\eta$ measures the deviation of the free surface from its equilibrium state and ${\bf v}$ is an approximation of the horizontal velocity taken at a prescribed depth. The constants $a$, $b$, $c$, and $d$ are modelling parameters satisfying the constraint $a+b+c+d=\frac{1}{3}$. 
 The regime of an asymptotic model is characterized by the nonlinearity parameter $\epsilon$ and the shallow water parameter $\mu$ defined by
\begin{align*}
     \epsilon=\frac{\mathtt{a}}{H} \ \ \ \  \  \text{ and }  \ \ \ \ \ \mu = \frac{H^2}{\lambda^2} ,
\end{align*}
where $\mathtt{a}$ is the typical amplitude of the wave, $H$ is the still water depth, and $\lambda$ is the typical wavelength. We will always assume that we are in the \emph{shallow water} regime\footnote{Strictly speaking, one should assume $0<\mu \le \mu_0$ for some $0<\mu_0 <1$. However, our analysis remains valid under the condition $0<\mu \le 1$.} 
\[ \mathcal{R}_{\mathrm{SW}} := \left\{ (\epsilon,\mu) :  0< \epsilon \le 1, \, 0<\mu \le 1 \right\} .\]
Moreover, an important sub-regime of $\mathcal{R}_{\mathrm{SW}}$ is the \emph{long wave} regime 
\[ \mathcal{R}_{\mathrm{LW}} := \left\{ (\epsilon,\mu) :   0<\mu \le 1, \, \epsilon \sim \mu  \right\} .\]

\begin{figure}[H]
    \centering
    \begin{tikzpicture}
\begin{axis}[x=1cm, y=1cm,
axis lines=middle,
axis x line shift=1.7, 
xlabel=\(x\),ylabel=\(z\),
x label style={anchor=north},
y label style={anchor=east},
   xticklabel=\empty,
xmin=0,xmax=12,ymin=-1.7,ymax=2.4,
xtick={},
clip=false,
domain=0:12, 
smooth,
ticks=none
]


\draw [black!80,densely dashed] (0.2,0) -- (12,0);
\draw [black!80,densely dashed] (0,0) -- (0.2,0);

\draw [aqua, very thick,  tension=0.6] plot [smooth] coordinates {(0,0)  (2,-0.2) (5,0.7) (7,-0.3)   (9.5,0.4)  (12,0)}  ;

\node[label=311:{$\epsilon\eta$}] (n2) at (2.5,0.9) {};
\node[label=311:{$-1$}] (n2) at (-0.75,-1.3) {};


\node[label=311:{$\mathtt{a}$}] (n2) at (4.85,0.6) {};
\draw [|-|] (5,-0.007) -- (5,0.72);

\node[label=311:{$\lambda$}] (n2) at (6.9,1.5) {};
\draw [|-|] (5,0.86) -- (9.5,0.86);

\node[label=311:{$H$}] (n2) at (9.35,-0.36) {};
\draw [|-|] (9.5,0) -- (9.5,-1.7);

\end{axis}
    \end{tikzpicture}
    \caption{The (blue) line illustrates the surface elevation $z= \epsilon\eta$ with a constant depth $z=-1$.}
\end{figure}
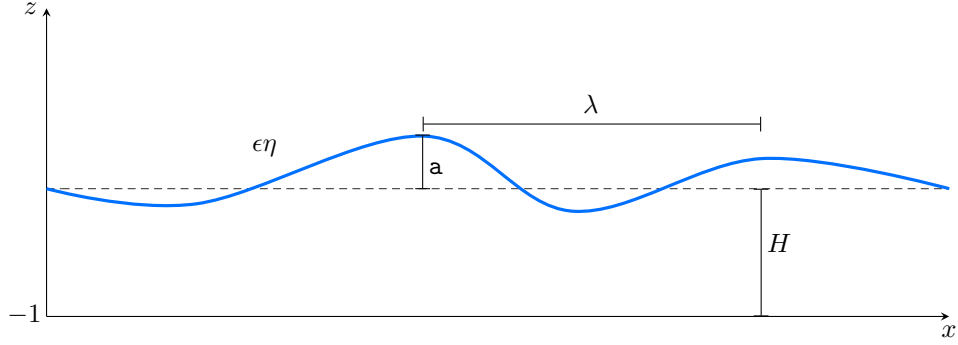
The Boussinesq systems \eqref{abcd} were derived from the full water waves system in the long wave regime with a precision  $\mathcal{O}(\epsilon^2 t)$ for $0\leq t\lesssim\epsilon^{-1}$ (see \cite{Bona2002,Bona2005,Lannes2013}).  Assuming further a unidirectional propagation of waves, one recovers the celebrated Korteweg-de Vries (KdV) and Benjamin-Bona-Mahony (BBM) equations (see \cite{Boussinesq1872,Korteweg1895} and \cite{Benjamin1972}). 

Comparing the KdV equation to the full water waves equations, the dispersion relation is too strong in high frequencies. As a consequence, it fails to capture important physical phenomena such as wave breaking and peaking waves. As an alternative to KdV, Whitham \cite{Whitham1967} proposed  the equation 
\begin{align}\label{eq:Whitham}
     \partial_t \eta + \sqrt{\mathcal{T}_\mu}(D) \partial_x \eta + \frac{\epsilon}{2} \partial_x (\eta^2) =0,
\end{align}
where the operator $ \sqrt{\mathcal{T}_\mu}(D)$ is the square root of the the Fourier multiplier $\mathcal{T}_\mu(D)$ defined in frequency by 
\begin{align}\label{eq:def_T_mu_xi}
    \mathcal{T}_\mu(\xi) = \frac{\tanh(\sqrt{\mu}|\xi|)}{\sqrt{\mu}|\xi|}.
\end{align}
Whitham introduced the equation formally, without the small parameters $\epsilon$ and $\mu$, by replacing the dispersion of KdV with the exact dispersion of the linearized water waves equations at a constant depth. The existence of shocks and Stokes waves, originally conjectured by Whitham, has been rigorously proved for \eqref{eq:Whitham}, respectively, in \cite{Hur2017,Saut2022}  and \cite{Ehrnstrom_Wahlen2019,Truong2022} (see also \cite{Hur2015} for a proof of the modulational instability of periodic waves). The equation with the parameters $\epsilon$ and $\mu$ was first introduced in \cite{Lannes2013}. In \cite{Klein2018}, the solutions of \eqref{eq:Whitham} were proven to stay close to those of KdV on the timescale $1/\epsilon$ when $(\epsilon,\mu) \in \mathcal{R}_{\mathrm{LW}}$. Later on, Emerald \cite{Emerald2021} proved the error estimate between the Whitham equation and the water waves system to have precision $\mathcal{O}(\mu\epsilon t)$ for $0\leq t\lesssim\epsilon^{-1}$ in the shallow water regime $ \mathcal{R}_{\mathrm{SW}}$. However, this is under the assumption of well-prepared initial conditions. In the general case, for bi-propagating waves, the precision is $\mathcal{O}((\mu\epsilon + \epsilon^2)t)$.

A similar procedure to that of Whitham can be applied to the Boussinesq systems \eqref{abcd}, leading to various Whitham-Boussinesq-type systems (also referred to as full-dispersion Boussinesq systems). In this paper, we focus on the one proposed by Hur and Pandey \cite{Hur2019},
\begin{align}\label{eq:whitham_boussinesq}
       \begin{dcases}
            \partial_t \eta  + \partial_x v + \epsilon \partial_x (\eta v)=0\\
        \partial_t v + \mathcal{T}_\mu(D) \partial_x \eta + \frac{\epsilon}{2} \partial_x (v^2) =0  
       \end{dcases}
\end{align}
with $\eta: \R \times \R \rightarrow \R$ and $v : \R \times \R \rightarrow \R$, and on its two-dimensional counterpart
\begin{align}\label{eq:whitham_boussinesq_d2}
       \begin{dcases}
            \partial_t \eta  + \nabla \cdot \mathbf{v} + \epsilon \nabla \cdot (\eta \mathbf{v})=0\\
        \partial_t \mathbf{v} + \mathcal{T}_\mu(D) \nabla \eta + \frac{\epsilon}{2} \nabla (|\mathbf{v}|^2) =\mathbf{0},
       \end{dcases}
\end{align}
where $\eta: \R^2 \times \R \rightarrow \R$ and $\mathbf{v} : \R^2 \times \R \rightarrow \R^2$ is curl-free. 
Here, the unknown variable $\eta$ in \eqref{eq:whitham_boussinesq} and \eqref{eq:whitham_boussinesq_d2} describes the surface elevation and the unknown variable $v$ (respectively,  $\mathbf{v}$ in \eqref{eq:whitham_boussinesq_d2}) is related to the fluid velocity. Other versions of these systems have been studied in, e.g., \cite{Klein2018,Kalisch2019,Paulsen2022,Dinvay2019,Hur2019,Dinvay2020,Tesfahun2024}. Moreover, the rigorous derivations of the Whitham-Boussinesq-type systems from the water waves system were proved  \cite{Emerald2021_siam} (see also \cite{Duchene2021} with bathymetry) with an error $\mathcal{O}(\mu\epsilon t)$ for $0\leq t\lesssim\epsilon^{-1}$, for $(\epsilon,\mu) \in \mathcal{R}_{\mathrm{SW}}$. On the other hand, it was also shown in \cite{Emerald2021_siam} that their strongly dispersive versions \eqref{abcd} have the precision of $\mathcal{O}((\mu\epsilon + \mu^2) t)$ for $0\leq t\lesssim\epsilon^{-1}$. As a consequence, the Whitham–Boussinesq systems provide a better approximation of the water waves system in the weakly nonlinear regime $\epsilon < \mu$. 

For this reason, we wish to investigate whether the lifespan of the solutions of the full-dispersion models \eqref{eq:Whitham}, \eqref{eq:whitham_boussinesq}, and \eqref{eq:whitham_boussinesq_d2} can be improved when the parameters $\mu$ and $\epsilon$ are decoupled. 

\subsection{Former well-posedness results}
The typical lifespan for solutions to a system of the form \eqref{general:model} is the \emph{hyperbolic} time $1/\epsilon$. This can be readily established by using energy methods for unidirectional equations and for systems with symmetric nonlinearity (see \cite{Bona2005}).

As a result, the initial value problem (IVP) associated with the Whitham equation \eqref{eq:Whitham} is well-posed in $H^s(\R)$, for $s>\frac{3}{2}$, on the timescale  $1/\epsilon$ uniformly in $\mu$. In the case $\mu=1$, a simple rescaling shows that this problem is equivalent to considering initial data of size $\epsilon$ and trying to extend the existence time beyond the hyperbolic timescale.   Assuming that  $(\epsilon,\mu)$ are in the \emph{small data} regime:
\[ \mathcal{R}_{\mathrm{SD}} := \left\{ (\epsilon,\mu) :   \mu \sim 1, \, 0<\epsilon \ll 1  \right\} ,\]
 Ehrnstr\"{o}m and Wang \cite{Ehrnstrom2022} obtained well-posedeness for the Whitham equation on the timescale $1/\epsilon^2$ in $H^N(\R)$ with $N\geq 3$ by combining hyperbolic techniques with a normal form transformation in the spirit of Shatah \cite{Shatah1985}. The proof in \cite{Ehrnstrom2022} is an adaptation of the methods introduced by Hunter, Ifrim, Tataru, and Wong \cite{Hunter2015} for the Burgers-Hilbert equation (see also \cite{Hunter2012}).  These results were also extended to the fractional KdV equation in \cite{Ehrnstrom2019}. Additionally, similar techniques have also been employed in more complex settings, notably for the Benjamin-Ono equation  \cite{Ifrim2019} and for the water waves system (see, e.g, \cite{Alazard2015,Hunter2016}).

The analysis becomes more delicate for systems with non-symmetric nonlinearities. One then has to find suitable symmetrizers and carefully track the small parameter in the commutator estimates. This was achieved by Saut and Xu \cite{Saut2012} for the Boussinesq systems \eqref{abcd}, where they obtained well-posedness on a large time of order $1/\epsilon$ in the long wave regime $\mathcal{R}_{\mathrm{LW}}$ (see also \cite{Ming2012,Saut2017,Saut2020_JDE,Burtea2016}). Moreover, in the small data regime $\mathcal{R}_{\mathrm{SD}}$, Saut and Xu \cite{Saut2020} extended the hyperbolic time of existence for a strongly dispersive Boussinesq system to $\epsilon^{-\frac43}$ by using a normal form transformation. 

The IVPs associated with the Whitham-Boussinesq systems  \eqref{eq:whitham_boussinesq} and \eqref{eq:whitham_boussinesq_d2} were proved by Paulsen \cite{Paulsen2022} to be well-posed in the shallow water regime $\mathcal{R}_{\mathrm{SW}}$ on the timescale $1/\epsilon$ for initial data in $H^s(\R^d) \times H^{s+\frac{1}{2}}(\R^d)$ with $s>1+\frac{d}{2}$ for $d=1,2$.\footnote{Actually, the function space has a dependence on $\mu$. See Definition \ref{def:solution_space}.} The proof relies on combining energy methods and symmetrizers in the spirit of \cite{Saut2012}. 

Next, we comment on related semi-linear Whitham-Boussinesq systems introduced by Dinvay, Dutykh, and Kalisch \cite{Dinvay2019_model}, where the nonlinearity in \eqref{eq:whitham_boussinesq} and \eqref{eq:whitham_boussinesq_d2} is relaxed by applying the nonlocal operator $\mathcal{T}_\mu(D)$. The system takes the form
 \begin{align}\label{eq:studied_in_Dinvay_Selberg_Tesfahun_d_1}
       \begin{dcases}
            \partial_t \eta  + \partial_x v + \epsilon  \mathcal{T}_\mu(D)\partial_x (\eta v)=0\\
        \partial_t v + \mathcal{T}_\mu(D) \partial_x \eta + \frac{\epsilon}{2}  \mathcal{T}_\mu(D)\partial_x (v^2) =0 
       \end{dcases}
\end{align}
in dimension one, and
\begin{align}\label{eq:studied_in_Dinvay_Selberg_Tesfahun_d_2}
    \begin{dcases}
            \partial_t \eta  + \nabla \cdot \mathbf{v} + \epsilon \mathcal{T}_\mu(D)\nabla \cdot (\eta \mathbf{v})=0\\
        \partial_t \mathbf{v} + \mathcal{T}_\mu(D) \nabla \eta + \frac{\epsilon}{2} \mathcal{T}_\mu(D)\nabla (|\mathbf{v}|^2) = \mathbf{0} 
       \end{dcases}
\end{align}
in dimension two. In \cite[Theorems 1 and 2]{Dinvay2020}, local-in-time well-posedness of \eqref{eq:studied_in_Dinvay_Selberg_Tesfahun_d_1} and \eqref{eq:studied_in_Dinvay_Selberg_Tesfahun_d_2} was proved at low regularity in the case $\epsilon=\mu=1$. In addition, for $d=1$ a global-in-time existence result was obtained in \cite[Theorem 3]{Dinvay2020} under a smallness condition on the initial data. Meanwhile, in \cite{Tesfahun2024}, the fourth author kept the parameters $0<\mu, \, \epsilon \le 1$ and proved the long-time well-posedness of the system \eqref{eq:studied_in_Dinvay_Selberg_Tesfahun_d_2} on the timescale $\mu^{\frac{3}{2}^-} \epsilon^{(-2)^+}$ in $H^s(\R^2) \times H^s(\R^2)$ for $s>\frac{1}{4}$, recovering the existence time $\epsilon^{(-2)^+}$ when $\mu=1$ without using a normal form method. However, the result provides only a time of existence of order $\epsilon^{-\frac12}$ in the long wave regime $\mathcal{R}_{\mathrm{LW}}$, which coincides with a former result in \cite{Linares2012} for a strongly dispersive Boussinesq system. The results in \cite{Linares2012,Dinvay2020,Tesfahun2024} were proved by combining a fixed point argument with dispersive techniques, which is possible due to the semi-linear nature of the system. 

Finally, the well-posedness of other Whitham-Boussinesq-type systems has been proven on the timescale $1/\epsilon$  in \cite{Paulsen2022,Emerald2022} (see also \cite{Dinvay2019,Dinvay2020_alone,Kalisch2019,Wang2021} for former well-posedness results on timescales of order $1$).

\subsection{Main results}
In this paper, we investigate the time of existence of the solutions to the Whitham equation \eqref{eq:Whitham}  and the Whitham-Boussinesq systems \eqref{eq:whitham_boussinesq} and \eqref{eq:whitham_boussinesq_d2} with dependence on the small parameters $(\epsilon,\mu)$. More precisely,  we establish the well-posedness of \eqref{eq:Whitham} and \eqref{eq:whitham_boussinesq} on the timescale of order $\mu^{\frac{1}{4}^-}\epsilon^{(-\frac{5}{4})^+}$ and of \eqref{eq:whitham_boussinesq_d2} on  the timescale of order $\mu^{\frac{1}{4}^-}\epsilon^{(-\frac{3}{2})^+}$. The proofs consist of compactness methods combined with Strichartz estimates. These techniques are robust and could be adapted to improve the lifespan of solutions for other equations and systems of the form \eqref{general:model}.

The result for the Whitham equation is summarized below. 
\begin{thm}\label{thm:wellposedness_Whitham}
    Let $s > \frac{13}{8}$, $\epsilon,\mu \in (0,1]$, and $\eta_0 \in H^s(\R)$.  Then there exists a positive time
    \begin{align} \label{T_eps_mu}
        T_{\epsilon,\mu}= c^{-1} \epsilon^{-1} \Big(\frac{\mu}{\epsilon}\Big)^{\frac{1}{4}^-} (\norm{\eta_0}_{H^s_x})^{(-\frac{5}{4})^+} 
    \end{align}
    such that \eqref{eq:Whitham} has a unique solution 
    \begin{align*}
       \eta \in C([0,T_{\epsilon,\mu}];H^s(\R)) \cap C^1([0,T_{\epsilon,\mu}];H^{s-1}(\R))
    \end{align*}
    satisfying the initial condition $\eta(0) = \eta_0$. Moreover,
    \begin{align*}
        \sup_{t\in[0,T_{\epsilon,\mu}]} \norm{\eta(t)}_{H^s_x} \lesssim \norm{\eta_0}_{H^s_x}.
    \end{align*}
    Further, there exists a neighborhood $\mathcal{U}$ of $\eta_0$ in $H^s(\R)$ such that the flow map
    \begin{align*}
        F_{\epsilon,\mu} : \mathcal{U} \rightarrow C([0,\frac{T_{\epsilon,\mu}}{2}];H^s(\R)), \ \ \ \ \eta_0 \mapsto \eta
    \end{align*}
    is continuous.
\end{thm}

\begin{remark}
  The result in \cite{Ehrnstrom2022} provides a larger time of existence in the small data regime $\mathcal{R}_{\mathrm{SD}}$, namely $T \sim \epsilon^{-2} \ge \epsilon^{-\frac{5}{4}}$. In comparison, our result incorporates the additional small parameter $\mu$ and lowers the regularity for the initial data to $s>\frac{13}{8}$. Moreover, in the long wave regime $\mathcal{R}_{\mathrm{LW}}$, we recover the hyperbolic timescale $\epsilon^{-1}$.  Although our result does not reach the timescale achieved in \cite{Ehrnstrom2022}, the proof is both robust and conceptually straightforward, relying solely on Strichartz estimates. Combining a normal form transformation with the techniques in this paper will be an objective for future work. 
  \end{remark}

To state our result for the Whitham-Boussinesq systems \eqref{eq:whitham_boussinesq} and \eqref{eq:whitham_boussinesq_d2}, we need to introduce the non-cavitation condition:
\begin{definition}[Non-cavitation condition]
    Let $d\in \{1,2\}$,  $s>\frac{d}{2}$, and $\epsilon\in (0,1]$. We say that the initial surface elevation $\eta_0 \in H^s(\R^d)$ satisfies the \textit{non-cavitation condition} if there exists a constant $h_0 \in (0,1)$ such that 
    \begin{align}\label{eq:noncav}
        1 + \epsilon\eta_0(x) \geq h_0
    \end{align}
    for all $x \in \R^d$. 
\end{definition}
The natural solution space for \eqref{eq:whitham_boussinesq}--\eqref{eq:whitham_boussinesq_d2} is defined as follows. 
\begin{definition}\label{def:solution_space}
    We define the solution space $V_\mu^s(\R^d)$ via the norm
    \begin{align*}
        \norm{(\eta,\mathbf{v})}_{V_\mu^s}^2  := \norm{\eta}_{H^s_x}^2 + \norm{\mathbf{v}}_{H^s_x}^2 + \sqrt{\mu} \norm{|D|^\frac{1}{2} \mathbf{v}}_{H^s_x}^2.
    \end{align*}   
\end{definition}
\begin{remark}
    We observe that 
    \begin{align}\label{eq:equivalent_norm_Vsmu}
        \norm{(\eta,\mathbf{v})}_{V_\mu^s}^2 \sim  \norm{\eta}_{H^s_x}^2 + \norm{\mathcal{T}_\mu^{-\frac{1}{2}}(D) \mathbf{v}}_{H^s_x}^2
    \end{align}
    (see \cite[Corollary 2.6]{Paulsen2022}).
\end{remark}
We are now in a position to state the result on the systems.
\begin{thm}\label{thm:wellposedness_WB}
    Let $d\in \{1,2\}$,  $s>1+\frac{5d}{8}$,
        and $\epsilon,\mu \in (0,1]$. Assume that the initial data $(\eta_0,\mathbf{v}_0) \in V^s_\mu(\R^d)$ 
    satisfy \eqref{eq:noncav} and that $\mathrm{curl}\:\mathbf{v}_0=0$ if $d=2$. Then there exists a positive time
    \begin{equation}\label{WBtime}
        T_{\epsilon,\mu,d}= c^{-1}\epsilon^{-\frac{d+3}{4}} \Big(\frac{\mu}{\epsilon}\Big)^{\frac{1}{4}^-} (\norm{(\eta_0,\mathbf{v}_0)}_{V^s_\mu})^{(-\frac{d+4}{4})^+},
    \end{equation}
    for some $c=c\bigl(d,s,h_0, 1+\epsilon\norm{(\eta_0,\mathbf{v}_0)}_{L^\infty_x\times L^\infty_x}\bigr) > 0$,
    such that \eqref{eq:whitham_boussinesq} and \eqref{eq:whitham_boussinesq_d2} admits a unique solution 
    \begin{align*}
        (\eta,\mathbf{v}) \in C([0,T_{\epsilon,\mu,d}];V_\mu^s(\R^d)) \cap C^1([0,T_{\epsilon,\mu,d}];V_\mu^{s-1}(\R^d))
    \end{align*}
    satisfying the initial conditions $\eta(0) = \eta_0$ and $\mathbf v(0) = \mathbf v_0$. Moreover,
    \begin{align*}
        \sup_{t\in[0,T_{\epsilon,\mu,d}]} \norm{(\eta,\mathbf{v})(t)}_{V_\mu^s} \lesssim \norm{(\eta_0,\mathbf{v}_0)}_{V_\mu^s},
    \end{align*}
    where the implicit constant may depend on the non-cavitation parameter $h_0$ and $1+\epsilon\norm{(\eta_0,\mathbf{v}_0)}_{L^\infty_x\times L^\infty_x}$. Furthermore, the flow map depends continuously on the initial data.
\end{thm}

\begin{remark}
In the one-dimensional case, we recover the result in  \cite[Theorem 1.11]{Paulsen2022} in the long wave regime $\mathcal{R}_{\mathrm{LW}}$, and moreover, we improve the lifespan of the solution whenever $\epsilon < \mu$. 
\end{remark}

\begin{remark}
In the two-dimensional case, we improve the result in \cite[Theorem 1.11]{Paulsen2022} when $0<\epsilon < \sqrt{\mu}$. It is worth noting that in the long wave regime $\mathcal{R}_{\mathrm{LW}}$, we obtain a lifespan of order $\epsilon^{-\frac{5}{4}}$. To our knowledge, results in which the lifespan exceeds the hyperbolic timescale $\epsilon^{-1}$ in the long wave regime appear to be new, even for the classical Boussinesq systems \eqref{abcd} (see, for example, \cite{Saut2012,Burtea2016} where hyperbolic techniques are used). This improvement might be a consequence of the dispersive nature of the zero-order wave system in \eqref{eq:whitham_boussinesq_d2}, which we exploit through Strichartz estimates\footnote{We refer to \cite{Sideris1997} for a long time existence result for the compressible Euler system which relates to the shallow water system corresponding to $\mu=0$ in \eqref{abcd} in the divergence free case.}.  Such property might also be true for some versions of the Boussinesq systems \eqref{abcd}. This would answer a question
raised by Jean-Claude Saut to the first author several years ago.
\end{remark}

\begin{remark}
    The water waves equations were proved to be well-posed on the timescale $1/\epsilon$ in the seminal work by Alvarez-Samaniego and Lannes \cite{Alvarez2008}. Since Theorems  \ref{thm:wellposedness_Whitham} and \ref{thm:wellposedness_WB} provide a larger time of existence in the weakly nonlinear regime $\epsilon<\mu$, one might naturally ask whether \eqref{eq:Whitham}, \eqref{eq:whitham_boussinesq}, and \eqref{eq:whitham_boussinesq_d2} continue to approximate the solution to the water waves equations over these extended timescales with the same precision as in \cite{Emerald2021,Emerald2021_siam}. To complete the justification, one also needs to prove well-posedness of the water waves equations on these larger timescales. These are prospects for future work.
\end{remark}


\begin{remark} \label{rem:Strich}
    To obtain the enhanced lifespan in Theorems \ref{thm:wellposedness_Whitham} and \ref{thm:wellposedness_WB}, we rely on Strichartz estimates for the linear part of the Whitham equation \eqref{eq:Whitham} that include the small parameter $\mu$.  These estimates were proved independently by Tesfahun  \cite{Tesfahun2024} and Melinand \cite{Melinand2024}. In Section \ref{sec:disp_ests}, we explain how they can be directly deduced from the case $\mu=1$ (see \cite{Dinvay2020}) by a simple scaling argument.  However, the main novelty of this paper is to combine new refined Strichartz estimates with the energy method to improve the existence time. 
\end{remark}

\subsection{Strategy and outline}\label{sec:strategy_outline}  
Broadly speaking, the proofs of Theorems \ref{thm:wellposedness_Whitham} and \ref{thm:wellposedness_WB} follow the same strategy, combining energy estimates with Strichartz estimates. We first illustrate the approach for the simpler case of the Whitham equation in Theorem \ref{thm:wellposedness_Whitham}, and then make a few remarks on the extensions needed to prove Theorem  \ref{thm:wellposedness_WB}.

The energy associated with the Whitham equation \eqref{eq:Whitham} is the square of the classical $H^s$-norm.
By differentiating in time and using \eqref{eq:Whitham}, we obtain
\begin{align} \label{en:est}
    \frac{\d}{\d t} \Big( \norm{\eta}_{H^s_x}^2\Big) &\le c  \epsilon \norm{\partial_x \eta}_{L^\infty_x} \norm{\eta}_{H^s_x}^2,
\end{align}
which yields by using a Gr\"{o}nwall argument 
\begin{align}\label{eq:est_eta_v_sec_1}
    \norm{\eta}_{L^\infty_T H^s_x}^2 \le \mathrm{e}^{c\epsilon T^{1/2} \norm{\partial_x \eta}_{L^2_T L^\infty_x} } \norm{\eta_0}_{H^s_x}^2.
\end{align}
The Sobolev embedding $H^{\frac12^+}(\mathbb R) \hookrightarrow L^{\infty}(\R)$ provides the classical time of existence $1/\epsilon$. 

To improve the lifespan, we will rely on the Strichartz estimates for the linear part of \eqref{eq:Whitham} to control $\norm{\partial_x \eta}_{L^2_T L^\infty_x}$, (see Remark \ref{rem:Strich} for some comments on the two-dimensional extension). If this estimate is applied directly to the Duhamel formulation of \eqref{eq:Whitham},  it yields a lifespan of order $  \epsilon^{-1} \min\big( (\mu/\epsilon)^{\frac{1}{3}}, (\mu/\epsilon)^{\frac{1}{7}} \big)$, which corresponds to  $\epsilon^{-\frac87}\mu^{\frac{1}{7}}$ when $\epsilon<\mu$. 

To fix this imbalance between $\mu$ and $\epsilon$, we introduce a family of dispersive Strichartz estimates depending on two free parameters $\omega$ and $\rho$ (see Theorem \ref{thm:refined_strichartz}).  The first one, $\omega$, characterizes the splitting of the function in high and low frequencies. We estimate the low-frequency part directly by using Bernstein's inequality. To estimate the high-frequency part, we first use a dyadic decomposition in frequency, and then split the time interval into small pieces of size $\rho(\lambda)$ depending on the dyadic frequency $\lambda$. The Strichartz estimate is then applied to each of these subintervals, and the resulting bounds are summed over all intervals. The technique in the proof of Theorem \ref{thm:refined_strichartz} is inspired from the works in \cite{Tataru2000}, \cite[Lemma 3.6]{Burq2004}, \cite[Lemma 2.1]{Koch2003}, \cite[Proposition 2.8]{Kenig2003}, 
\cite[Proposition 2.3]{Linares2014},  and \cite[Lemma 4.1]{Pilod2021}. 

Accordingly, we obtain the enhanced lifespan in Theorem \ref{thm:wellposedness_Whitham} by choosing the parameters $\omega$ and $\rho$ such that the small parameters $\epsilon$ and $\mu$ are optimally balanced in the estimates (see Corollary \ref{cor:refined_strichartz_d_1}). This gives us, for solutions of \eqref{eq:Whitham},
\begin{align}\label{eq:est_V_cal}
    \epsilon T^\frac{1}{2}\norm{\partial_x \eta}_{L^2_T L^\infty_x} &\lesssim \epsilon \mu^{(-\frac{1}{5})^+ }   T^{\frac{4}{5}^+}  \norm{\eta}_{L^\infty_TH^s_x} +  \epsilon^2  \mu^{(-\frac{2}{5})^+}  T^{\frac{8}{5}^+} \norm{\eta}^2_{L^\infty_TH^s_x}
\end{align}
for $s> \frac{13}{8}$ (see Proposition \ref{prop:boundedness_whitham}). Combining the estimates \eqref{eq:est_eta_v_sec_1} and \eqref{eq:est_V_cal} we may apply a standard bootstrap argument to obtain a priori bounds on the energy on the timescale $T_{\epsilon,\mu}$ as defined in \eqref{T_eps_mu} (see Lemma \ref{lem:bootstrap_whitham}). This gives us the existence of a solution $\eta \in C([0,T_{\epsilon,\mu}]; H^s({\mathbb R}))$. The uniqueness and continuous dependence on the initial data will be obtained by utilizing the existing local well-posedness result (see Remarks \ref{rem:w_rem1}--\ref{rem:w_rem3} for more details).   

\indent 
The approach explained above can also be extended to systems. This is illustrated in the proof of Theorem \ref{thm:wellposedness_WB}. Here, we will focus on the proof of the two-dimensional case, which is more involved. For this, we need to refine the energy estimates in \cite[Proposition 3.3]{Paulsen2022} to obtain an estimate in the spirit of \eqref{en:est} with the $L^\infty$-norm of the derivatives of the solutions appearing on the right-hand side (see Proposition \ref{prop:energy_ests_d_2}). This will allow us to use the refined Strichartz estimate on the diagonalized version of \eqref{eq:whitham_boussinesq_d2} (see Proposition \ref{prop:boundedness_eta_v_d_2}) as for the Whitham equation. 

 The energy functional is slightly changed in comparison to \cite[Definition 1.18]{Paulsen2022}, as it includes the nonlocal operator $\mathcal{T}_{\mu}^{-\frac{1}{2}}(D)$ (instead of $J^\frac{1}{2}_\mu=\jbrack{\sqrt{\mu}D}^{\frac12}$):
 \begin{definition}\label{def:energy}
    Let $\epsilon,\mu\in (0,1]$ and let $J^s$ be the Bessel potential of order $-s$. We define the energy functional associated with \eqref{eq:whitham_boussinesq_d2} as
    \begin{align*}
        E_s(\eta,\mathbf{v}) = \int_{\R^2} \Big\{ \Big( J^s \eta \Big)^2 + (1+\epsilon\eta)\Big|J^s \mathcal{T}_{\mu}^{-\frac{1}{2}}(D) \mathbf{v} \Big|^2 \Big\} \d x.
\end{align*}
\end{definition}

This, in turn, simplifies the calculations. The energy in Definition \ref{def:energy} requires new commutator estimates (see Lemmas \ref{lem:commutator_est_tilbert} and \ref{lem:commutator_inv_tilbert}). With this energy and the commutator estimates, we can extend the techniques used for the proof of Theorem \ref{thm:wellposedness_Whitham} (explained above) to prove Theorem \ref{thm:wellposedness_WB}.

Finally, we refer to \cite{Melinand2024b} where classical Strichartz estimates are combined with the energy method to study the rigid lid in shallow water of several Boussinesq system.

\medskip
\indent The paper is organized as follows. Section \ref{sec:RefinedGeneric} is devoted to proving a generic refined Strichartz estimate (Theorem \ref{thm:refined_strichartz}). In Section \ref{sec:disp_ests}, we first prove Strichartz estimates including the small parameter $\mu$ (Lemma \ref{strichartz:lemma}), which together with Theorem \ref{thm:refined_strichartz} will allow us to derive new refined Strichartz estimates for \eqref{eq:Whitham}, \eqref{eq:whitham_boussinesq}, and \eqref{eq:whitham_boussinesq_d2} (Corollary \ref{cor:refined_strichartz_d_1}). Theorem \ref{thm:wellposedness_Whitham}  is proved in Section \ref{sec:Whitham} and Theorem \ref{thm:wellposedness_WB} in Section \ref{sec:WB_d_2}. In Sections \ref{sec:energy_whitham}--\ref{sec:refined_str_whitham}
(respectively, Sections \ref{sec:commutator_ests}--\ref{sec:refined_str_wb}), we will prove tools needed to prove Theorems \ref{thm:wellposedness_Whitham} (respectively, Theorem \ref{thm:wellposedness_WB}). Finally, we combine these results in Section \ref{sec:proof_of_Whitham} (respectively,  Section \ref{sec:proof_main_res_two_dims}) to prove Theorem \ref{thm:wellposedness_Whitham} (respectively, Theorem \ref{thm:wellposedness_WB}).

\subsection{Notation}

\begin{itemize}
 \item We let $c$ denote a positive constant which may change from line to line. 
\item We use the standard notations $\sim$, $\lesssim$, and $\gtrsim$ involving implicit constants that are independent of $\mu$, $\epsilon$, and the time $T$.
\item $\lceil s \rceil$ for $s\in \R$ denotes the smallest integer larger than or equal to $s$.
\item We denote by $[A,B] = AB -BA$ the commutator between two operators $A$ and $B$. 
\item We use $f\ast g (x) = \int_{\R^d} f(x-y)g(y) \d y$ for any $x\in \R^d$ to denote the convolution in space of two functions $f$ and $g$.
\item Let $L^p(\R^d)$ denote the Lebesgue $p$-integrable functions on $\R^d$ equipped with the norm $\norm{v}_{L^p_x} = (\int_{\R^d} |v|^p \d x)^{1/p}$ for all $p\in [1,\infty)$ and the usual modification for $p=\infty$. Additionally, for any $f,g \in L^2(\R^d)$, we denote the scalar product by $\scalarprod{f}{g}_{L^2_x} = \int_{\R^d} f(x) \overline{g(x)} \d x$.
 \item For simplicity in notation, we denote $L^p_T X:= L^p((0,T);X)$ for any Banach space $X$ (of functions on $\R^d$), $T>0$, and $p \in [1,\infty]$. For $p\in [1,\infty)$,  $L^p_T X$ is equipped with the norm $\norm{f}_{L^p_T X_x} = \Big( \int_0^T \norm{f(\cdot,t)}_{X_x}^p \d t \Big)^{1/p}$ and with the usual modifications for $p=\infty$. The subscript on $X$ is included to make a clear distinction between the norms in time and space. 
\item Let $\mathcal{S}(\R^d)$ denote the Schwartz space.
\item For any $u \in \mathcal{S}(\R^d)$, let $\mathcal{F} u =  \widehat{u} = (u)^{\wedge} $ and $\mathcal{F}^{-1} u = \widecheck{u} =  (u)^{\vee}$ denote the Fourier transform and the inverse Fourier transform, respectively. 
  \item Let $m:\R\rightarrow \R$ be a smooth function, then we denote by $m(D)$ the Fourier multiplier defined by the Fourier transform as $\widehat{m(D)f}(\xi) =m(\xi) \widehat{f}(\xi)$.
\item We denote by $\mathbf{R} = (R_1,R_2)$ the Riesz transformation, where $R_j$ is defined via the Fourier transform as $R_jf = ( -i\frac{\xi_j}{|\xi|} \widehat{f} )^\vee$ for  $j=1,2$ and $f \in \mathcal{S}(\R^2)$.
\item We let $D=-i \partial_x$ for $d=1$ and $D=-i \nabla$ for $d=2$. 
\item For any $s\in \R$ and $\xi \in \R^d$,  the multiplier $\widehat{|D|^sf}(\xi) = |\xi|^s \widehat{f}(\xi)$ is referred to as the Riesz potential of order $-s$. If $d=s=1$, we may write $|D| = \mathcal{H}\partial_x$, where $\mathcal{H}$ is the Hilbert transform defined by the Fourier transform as $\widehat{\mathcal{H}f}(\xi) = -i \mathrm{sgn}(\xi) \widehat{f}(\xi)$. Similarly, for $d=2$ and $s=1$, we may write $|D| = \mathbf{R} \cdot \nabla$.
\item We use the Japanese bracket $\jbrack{\xi}:= (1+|\xi|^2)^{1/2}$ for $\xi \in \R^d$.
\item For any $s\in \R$, $J^s = \jbrack{D}^s$ denotes the Bessel potential of order $-s$. For $\mu \in (0,1]$ and $s\in \R$ we write $J^s_\mu = \jbrack{\sqrt{\mu}D}^s$.
\item The Sobolev space $H^s(\R^d)$ has the norm
 $\norm{f}_{H^s_x} = \norm{J^s f }_{L^2_x}$.
 \item Let $\chi \in C^\infty_c((-2,2))$ such that $\chi(s) = 1$ on $|s| \leq 1$.
 \item We define $\beta(s) := \chi (s) - \chi(2s)$, and the scaled version $\beta_\lambda (s) := \beta(s/\lambda)$ for $\lambda>0$ with the support  $\mathrm{supp} \: \beta_\lambda \subset \{ s\in \R : \lambda/2 \leq |s| \leq \lambda \}$.  
\item For each number $\lambda >0$ and $\xi \in \R^d$, we define the frequency projection $P_\lambda$ by
    $\widehat{P_\lambda f}(\xi) =  \beta_\lambda (|\xi|) \widehat{f}(\xi)$.
\item Let $\alpha > 0$ and $\xi \in \R^d$, then we denote $P_{\leq \alpha} f = \big( \chi(|\xi|/\alpha) \widehat{f} \big)^\vee $ and $P_{> \alpha} f = \big((1- \chi(|\xi|/\alpha)) \widehat{f} \big)^\vee $ so that $f = P_{\leq \alpha} f + P_{>\alpha} f$.
\item For $f=\{f_j\}_{j\in J}$, we denote the norm of $\ell^2_{j\in J}$ by $ \norm{f}_{\ell^2_{j\in J}}
  = \big(\sum_{j\in J} |f_j|^2 \big)^{1/2}$.
\end{itemize}


\section{Dispersive estimates for an inhomogeneous equation} \label{sec:RefinedGeneric}
To extend the time of existence of a general model \eqref{general:model}, we consider a generic dispersive linear inhomogeneous IVP of the form
\begin{align}\label{eq:whitham_F_1}
        \begin{dcases}
            i\partial_t u + M(D) u = F\\
            u(x,0) = f(x)
        \end{dcases}
    \end{align}
for a given symbol $M(\xi) \in \R$, whose solution $u(x,t)$, $(x,t)\in \R^d \times \R$, is given by
\begin{align*}
    u(t) =  S(t) f(x) + i\int_0^t S(t-s) F(s) \d s
\end{align*} 
where $S(t) = e^{i t M(D)}$. Fixing $(q,r) \in [2,\infty] \times [2,\infty)$, we assume that the frequency-localized estimate
\begin{equation}\label{GenericLocalised}
   \norm{S(t) P_{\lambda} f}_{L_t^q(\R; L_x^r(\R^d))} \lesssim A(\lambda)^{\frac{1}{2}-\frac{1}{r}} \norm{P_\lambda f}_{L^2_x(\R^d)}
\end{equation}
holds for all $\lambda >0$, for given $A(\lambda) \in (0,\infty)$. With the estimate \eqref{GenericLocalised}, we derive a family of dispersive Strichartz estimates depending on two parameters $\omega$ and $\rho$, which can be chosen freely (see Theorem \ref{thm:refined_strichartz}).  The parameter $\omega$ represents the splitting of the function in high and low frequencies. For the high-frequency part, we divide the time interval into small pieces of size $\rho(\lambda)$, where $\lambda$ is the dyadic frequency. Then, for suitable choices of $(\omega,\rho)$, our generalization can be applied to these types of IVPs to obtain \textit{refined} Strichartz estimates (see, e.g., Corollary \ref{cor:refined_strichartz_d_1}).
More precisely, we have the following result.

\begin{thm}\label{thm:refined_strichartz}
    Let $d\in \N\backslash\{0\}$ and $T>0$. Assume that $u : \R^d \times \R \rightarrow \R$ solves
    \eqref{eq:whitham_F_1}
    for a given symbol $M(\xi) \in \R$,  and that for some fixed pair $(q,r) \in [2,\infty] \times [2,\infty)$ the associated unitary group $S(t) = e^{itM(D)}$ satisfies 
 \eqref{GenericLocalised}
for all $\lambda >0$ dyadic, for given $A(\lambda) \in (0,\infty)$. Then
\begin{align}\label{GenericRefined}
  \norm{u}_{L_T^2 L_x^\infty}
  &\lesssim
 \omega^{d(\frac{1}{2}-\frac{1}{\Tilde{r}})} T^{\frac{1}{2}} \norm{J^{\Tilde{\theta}} P_{\le \omega} u}_{L_T^\infty L_x^2}
 \\ \notag
 &+ \omega^{-\gamma} T^{\frac{1}{2}} \norm{J^{\theta+\gamma}  \rho(\abs{D})^{-\frac{1}{q}} A(\abs{D})^{\frac{1}{2}-\frac{1}{r}} u}_{L_T^\infty L_x^2}
  \\ \notag
  &+ \norm{J^\theta \rho(\abs{D})^{1-\frac{1}{q}} A(\abs{D})^{\frac{1}{2}-\frac{1}{r}} F}_{L_T^2 L_x^2}
\end{align}
for any $\Tilde{r}\in [2,\infty)$, $\Tilde{\theta} > d/\Tilde{r}$, $\theta > d/r$, and arbitrarily small $\gamma > 0$. Here we are free to choose $\omega > 0$ and $\rho(\lambda) \in (0,T]$.
\end{thm}
\begin{remark}
    The proof is inspired by the works in \cite{Tataru2000}, \cite[Lemma 3.6]{Burq2004}, \cite[Lemma 2.1]{Koch2003}, \cite[Proposition 2.8]{Kenig2003}, \cite[Proposition 2.3]{Linares2014},  and \cite[Lemma 4.1]{Pilod2021}.  Compared to these results, we are concerned with enhanced lifetime. This requires a different splitting of the high and low frequencies $\omega$ and size of the small time intervals $\rho(\lambda)$ to obtain Corollary \ref{cor:refined_strichartz_d_1}. Here, both $\omega$ and $\rho$ will be chosen to depend on the small parameter $\mu$, which will result in the enhanced lifespans in Theorems \ref{thm:wellposedness_Whitham} and \ref{thm:wellposedness_WB}. 
\end{remark}


Before we can prove Theorem \ref{thm:refined_strichartz}, we recall Bernstein's lemma.
\begin{lem}[Bernstein's lemma] For $\alpha > 0$ , we have the estimate
\[
  \norm{P_{\leq \alpha} f}_{L^q(\R^d)} \lesssim \alpha^{d(\frac{1}{p}-\frac{1}{q})} \norm{P_{\leq \alpha}  f}_{L^p(\R^d)}
\]
for $1 \le p \le q \le \infty$.
\end{lem}

\begin{proof}[Proof of Theorem \ref{thm:refined_strichartz}]
    For a number $\omega > 0$ to be chosen, we split into low and high frequencies:
\begin{align*}
    u = P_{\le \omega} u + \sum_{\lambda > \omega} P_\lambda u =: u_1 + u_2 ,
\end{align*}
giving us
\begin{align}\label{eq:decomp_u_ests}
    \norm{u}_{L_T^2 L_x^\infty} \leq \norm{u_1}_{L_T^2 L_x^\infty} +\norm{u_2}_{L_T^2 L_x^\infty}.
\end{align}
\subsubsection*{\underline{Control of $\norm{u_1}_{L_T^2 L_x^\infty}$ in \eqref{eq:decomp_u_ests}}}
By a Sobolev embedding and Bernstein's lemma in $x$, and H\"{o}lder's inequality in time, the low-frequency part is estimated by
\begin{equation}\label{GenericRefinedLow}
  \norm{u_1}_{L_T^2 L_x^\infty} \lesssim  \norm{ J^{\Tilde{\theta}} P_{\le \omega} u}_{L_T^2 L_x^{\Tilde{r}}} \lesssim \omega^{d(\frac{1}{2}-\frac{1}{\Tilde{r}})} T^{\frac{1}{2}} \norm{J^{\Tilde{\theta}} P_{\le \omega} u}_{L_T^\infty L_x^2}
\end{equation}
for any $\Tilde{r}\in [2,\infty)$ and $\Tilde{\theta} > d/\Tilde{r}$.
\subsubsection*{\underline{Control of $\norm{u_2}_{L_T^2 L_x^\infty}$ in \eqref{eq:decomp_u_ests}}}
By a Sobolev embedding, the Littlewood-Paley square function theorem, and Minkowski's integral inequality, we obtain
\[
  \norm{u_2}_{L_T^2 L_x^\infty} \lesssim  \norm{J^\theta u_2}_{L_T^2 L_x^r} 
  \lesssim  \norm{\sqrt{ \sum_{\lambda > \omega} \Abs{ P_\lambda J^\theta u }^2 } }_{L_T^2 L_x^r}
  \lesssim \sqrt{ \sum_{\lambda > \omega} \norm{ P_\lambda J^\theta u }_{L_T^2 L_x^r}^2 }  
\]
for any $\theta > d/r$. Fix $\lambda > \omega$ and write $[0,T] = \bigcup_{j \in J} I_j$ where $I_j = [a_j,b_j]$ are almost disjoint, $\abs{I_j} = \rho$ and $\# J = T/\rho$, where $\rho=\rho(\lambda) \in (0,T]$ remains to be chosen. Then, applying H\"older's inequality in time, we estimate
\[
  \norm{ P_\lambda J^\theta u }_{L_T^2 L_x^r}
  = \norm{ \norm{ P_\lambda J^\theta u }_{L_{I_j}^2 L_x^r } }_{\ell^2_{j\in J}}
  \le  \rho^{\frac{1}{2}-\frac{1}{q}}  \norm{\norm{ P_\lambda J^\theta u }_{L_{I_j}^q L_x^r} }_{\ell^2_{j\in J}}.
\]
On $I_j$ we write
\[
  u(t) = S(t-a_j) u(a_j) + i \int_{a_j}^t S(t-s) F(s) \, ds
\]
and apply \eqref{GenericLocalised} to obtain
\begin{align*}
  \norm{ P_\lambda J^\theta u }_{L_{I_j}^q L_x^r}
  &\lesssim
  A(\lambda)^{\frac{1}{2}-\frac{1}{r}} \left( \norm{P_\lambda J^\theta u(a_j)}_{L_x^2} + \int_{a_j}^t \norm{P_\lambda J^\theta F(s)}_{L_x^2} \, ds \right)
  \\
  &\lesssim
  A(\lambda)^{\frac{1}{2}-\frac{1}{r}} \left( \norm{P_\lambda J^\theta u}_{L_{I_j}^\infty L_x^2} + \rho^{\frac{1}{2}} \norm{P_\lambda J^\theta F}_{L_{I_j}^2 L_x^2} \right).
\end{align*}
Taking the $\ell^2_{j\in J}$ norm, we conclude that
\begin{align*}
  \norm{ P_\lambda J^\theta u }_{L_T^2 L_x^r}
  &\lesssim  \rho^{\frac{1}{2}-\frac{1}{q}}   A(\lambda)^{\frac{1}{2}-\frac{1}{r}} \left( \sqrt{\# J} \norm{P_\lambda J^\theta u}_{L_T^\infty L_x^2} + \rho^{\frac{1}{2}} \norm{P_\lambda J^\theta F}_{L_T^2 L_x^2} \right)
  \\
  &\lesssim  \rho^{\frac{1}{2}-\frac{1}{q}}   A(\lambda)^{\frac{1}{2}-\frac{1}{r}} \left( T^{\frac{1}{2}} \rho^{-\frac{1}{2}} \norm{P_\lambda J^\theta u}_{L_T^\infty L_x^2} + \rho^{\frac{1}{2}} \norm{P_\lambda J^\theta F}_{L_T^2 L_x^2} \right).
\end{align*}
Finally, squaring and summing over dyadic $\lambda > \omega$ we obtain
\begin{align*}
   \norm{u_2}_{L_T^2 L_x^\infty}  &\lesssim T^{\frac{1}{2}} \left( \sum_{\lambda > \omega} \lambda^{-2\gamma} \right)^{\frac{1}{2}}  \sup_{\lambda} \norm{P_\lambda J^{\theta+\gamma}  \rho(\abs{D})^{-\frac{1}{q}} A(\abs{D})^{\frac{1}{2}-\frac{1}{r}} u}_{L_T^\infty L_x^2}
  \\
  &+ \norm{P_\lambda J^\theta \rho(\abs{D})^{1-\frac{1}{q}} A(\abs{D})^{\frac{1}{2}-\frac{1}{r}} F}_{L_T^2 L_x^2}
\end{align*}
for any $\gamma > 0$. Thus
\begin{align}\label{GenericRefinedhigh}
   \norm{u_2}_{L_T^2 L_x^\infty} &\lesssim \omega^{-\gamma} T^{\frac{1}{2}} \norm{J^{\theta+\gamma} \rho(\abs{D})^{-\frac{1}{q}} A(D)^{\frac{1}{2}-\frac{1}{r}} u}_{L_T^\infty L_x^2}
  \\ \notag
  &+ \norm{P_\lambda J^\theta \rho(\abs{D})^{1-\frac{1}{q}} A(\abs{D})^{\frac{1}{2}-\frac{1}{r}} F}_{L_T^2 L_x^2}
\end{align}
for $\theta > \frac{d}{r}$ and arbitrarily small $\gamma > 0$.\\
\indent
Combining \eqref{GenericRefinedLow} and \eqref{GenericRefinedhigh} with \eqref{eq:decomp_u_ests} concludes the proof.
\end{proof}

\section{Dispersive estimates for linear Whitham-type equations}\label{sec:disp_ests}
In this section, we will first deduce frequency-localized dispersive decay estimates for linear Whitham-type equations (Lemma \ref{disp:lemma}) given by
\begin{equation}\label{Lin:1d}
 \begin{dcases}
     i \partial_t u - \frac{1}{\sqrt{\mu}}m_d(\sqrt{\mu}D) u =0\\
     u(x,0) = f(x)
 \end{dcases} 
\end{equation}
for $(x,t) \in \R^d \times \R$, where the Fourier multiplier $m_d(D)$ is defined in frequency by
\begin{align}\label{eq:m_1}
  m_1(\xi) &= \xi \sqrt{\mathcal{T}_1(\xi)}  = \xi \sqrt{ \frac{\tanh (\abs{\xi})}{\abs{\xi}} } & (\xi \in \R)
  \intertext{and}
  m_2(\xi) &= \abs{\xi} \sqrt{\mathcal{T}_1(\xi)} = \abs{\xi} \sqrt{ \frac{\tanh (\abs{\xi})}{\abs{\xi}} } & (\xi \in \R^2).
  \label{eq:m_2}
\end{align}
Here, the associated propagators are
\begin{equation}\label{eq:unitary_group}
     S_{\mu,d}(t) f = e^{i (t / \sqrt{\mu}) m_d(\sqrt{\mu} D)} f.
\end{equation}
The frequency-localized dispersive decay estimates will then be used to derive frequency-localized Strichartz estimates (see Lemma \ref{strichartz:lemma}). With the Strichartz estimates, we may utilize Theorem \ref{thm:refined_strichartz} and choose refinements of the time interval and frequencies such that we optimally balance the small parameters $(\epsilon,\mu)$. This will give us the desired refined Strichartz estimates (Corollary \ref{cor:refined_strichartz_d_1}), which are crucial in obtaining the improved lifespan (see Proposition \ref{prop:boundedness_whitham}  and Lemma \ref{lem:bootstrap_whitham} for $d=1$, and Proposition \ref{prop:boundedness_eta_v_d_2} and Lemma \ref{lem:bootstrap_wb} for $d=2$).

\subsection{Dispersive decay for linear Whitham-type equations}
\label{sec:local_disp}
Before we can derive the Strichartz estimates, we need frequency-localized dispersive decay estimates for the unitary group \eqref{eq:unitary_group} including the small parameter $\mu\in (0,1]$. For $d=2$, these estimates were already provided in \cite[Lemma 4]{Tesfahun2024} (stated in Lemma \ref{disp:lemma}). It therefore remains to derive these estimates for $d=1$. Introducing a scaling operator, we will observe that it suffices to show these estimates for $\mu=1$. 

Indeed, fixing the dimension $d \in \N$ and a scaling factor $\alpha > 0$, we define the scaling operator $\sigma_\alpha$ by
\[
  \sigma_\alpha f(x) = \alpha^d f(\alpha x) \qquad (x \in \R^d)
\]
for any $f \colon \R^d \to \R$. This operator has the following properties:
\begin{align*}
   \sigma_\alpha^{-1} &= \sigma_{\alpha^{-1}},
  &\partial_x \sigma_\alpha f &= \alpha \sigma_\alpha \partial_x f,
  \\
  \sigma_\alpha (fg) &= \alpha^{-d} (\sigma_\alpha f)(\sigma_\alpha g),
    &\norm{\sigma_\alpha f}_{L^p(\R^d)} &= \alpha^{d(1-1/p)}  \norm{f}_{L^p(\R^d)},
   \\
  \mathcal F\{ \sigma_\alpha f\}(\xi) &= \widehat f(\alpha^{-1}\xi).
\end{align*}
From the last property, we deduce the rescaling rule for Fourier multipliers,
\begin{equation*}
  \sigma_\alpha M\left( \alpha D \right) f =  M(D) \sigma_\alpha f
\end{equation*}
for any symbol $M(\xi)$. From the properties of the rescaling operator, one has
\begin{align}\label{eq:scaling_S}
     S_{\mu,d}(t) P_\lambda f = \sigma_{1/\sqrt{\mu}} S_{1,d}\left( \frac{t}{\sqrt{\mu}} \right) P_{\sqrt{\mu} \lambda} \sigma_{\sqrt{\mu}} f
\end{align}
permitting a reduction to the case $\mu=1$ in the following lemma.

\begin{lem}[Frequency-localized dispersive estimates]
\label{disp:lemma} 
Let $d\in\{1,2\}$ and $\mu \in (0,1]$.
Then the dispersive decay bound
\begin{equation}\label{dispersive:1d}
   \norm{S_{\mu,d}(t) P_\lambda f}_{L^\infty_x(\R^d)} \lesssim \mu^{-\frac{1}{2}} \abs{t}^{-\frac{d}{2}} \lambda^{\frac{d}{2}-1} \jbrack{\sqrt{\mu} \lambda}^{\frac{d}{4}+1} \norm{f}_{L^1_x(\R^d)}
\end{equation}
holds for $\lambda > 0$ and $t\in \R\backslash\{0\}$. Here, $S_{\mu,d}(t)$ is defined in \eqref{eq:unitary_group}.
\end{lem}
\begin{remark}\label{remark:decay_remark}
    The bound \eqref{dispersive:1d} for $d=2$ was proved in \cite[Lemma 4]{Tesfahun2024}. Alternatively one could apply \eqref{eq:scaling_S} to reduce to the case $\mu=1$, where the bound was proved in \cite[Lemma 10]{Dinvay2020} (for $\lambda > 1$, to be precise, but the proof applies for $\lambda \le 1$).
\end{remark}
\begin{proof}
In view of the preceding remark we only need to consider the case $d=1$, and by \eqref{eq:scaling_S} we may take $\mu=1$.

Writing
\[
  S_{1,1}(t) P_\lambda f = I_{\lambda,1}(t) * f,
\]
where
\[
  I_{\lambda,1}(x,t) = \int_{\R} e^{ix \xi + it m_1(\xi)} \beta(\abs{\xi}/\lambda) \d\xi,
\]
one has
\begin{equation}\label{eq:Sd1}
    \norm{S_{1,1}(t) P_\lambda f}_{L^\infty_x(\R)} \le \norm{I_{\lambda,1}(x,t)}_{L^\infty_x(\R)} \norm{f}_{L^1_x(\R)},
\end{equation}
so it suffices to check that
\begin{equation}\label{Ibound:1d}
   \Abs{I_{\lambda,1}(x,t)} \lesssim \abs{t}^{-\frac{1}{2}} \lambda^{-\frac{1}{2}} \jbrack{\lambda}^{\frac{5}{4}}.
\end{equation}
The bound \eqref{Ibound:1d} was proved in \cite[Lemma 10]{Dinvay2020} for $\lambda > 1$, but the proof also applies for $\lambda \le 1$, as one can check. (As an alternative we remark that a simplified argument that appeals to the van der Corput lemma is possible).
With \eqref{eq:scaling_S} and the estimates \eqref{eq:Sd1}--\eqref{Ibound:1d}, we obtain 
\begin{align*}
    \norm{S_{\mu,1}(t) P_\lambda f}_{L^\infty_x(\R)} &= \frac{1}{\sqrt{\mu}}\norm{S_{1,1}\left( \frac{t}{\sqrt{\mu}} \right) P_{\sqrt{\mu} \lambda} \sigma_{\sqrt{\mu}} f}_{L^\infty_x(\R)}\\
    &\lesssim \frac{1}{\sqrt{\mu}} \Abs{\frac{t}{\sqrt{\mu}}}^{-\frac{1}{2}} (\sqrt{\mu}\lambda)^{-\frac{1}{2}} \jbrack{\sqrt{\mu}\lambda}^{\frac{5}{4}} \norm{\sigma_{\sqrt{\mu}} f}_{L^1_x(\R)}\\
    &= c \mu^{-\frac{1}{2}} \abs{t}^{-\frac{1}{2}} \lambda^{-\frac{1}{2}} \jbrack{\sqrt{\mu}\lambda}^{\frac{5}{4}} \norm{f}_{L^1_x(\R)} 
\end{align*}
 as desired.
\end{proof}

\subsection{Refined Strichartz estimates}\label{sec:refined_strichartz}
From the decay bounds in Lemma \ref{disp:lemma}, one obtains immediately, via the standard $T T^*$-argument, the following frequency-localized Strichartz estimates.

\begin{lem}[Frequency-localized Strichartz estimates] \label{strichartz:lemma} Let $d \in \{1,2\}$ and $\mu \in (0,1]$. Assume that $2 < q \le \infty$ and $2 \le r \le \infty$ satisfy
\begin{equation}\label{eq:hyp_qr}
  \frac{2}{q} = d  \left( \frac{1}{2} - \frac{1}{r} \right).
\end{equation}
Then 
\begin{equation}\label{strichartz:1d}
   \norm{S_{\mu,d}(t) P_\lambda f}_{L^q(\R ;L^r(\R^d))} \lesssim \left( \mu^{-\frac{1}{2}} \lambda^{\frac{d}{2}-1} \jbrack{\sqrt{\mu} \lambda}^{\frac{d}{4}+1} \right)^{\frac{1}{2}-\frac{1}{r}} \norm{P_\lambda f}_{L^2(\R^d)}
\end{equation}
holds for all $\lambda > 0$, where $S_{\mu,d}(t)$ is defined in \eqref{eq:unitary_group}.
\end{lem}

With the frequency-localized estimates in Lemma \ref{strichartz:lemma}, we may utilize Theorem \ref{thm:refined_strichartz} to derive refined Strichartz estimates.  These estimates will be the key ingredient in obtaining the enhanced lifespan (see Proposition \ref{prop:boundedness_whitham}  and Lemma \ref{lem:bootstrap_whitham} for $d=1$, and Proposition \ref{prop:boundedness_eta_v_d_2} and Lemma \ref{lem:bootstrap_wb} for $d=2$). 

\begin{cor}[Refined Strichartz estimates]\label{cor:refined_strichartz_d_1}
    Let $d\in\{1,2\}$, $\mu \in(0,1]$,  and $T>0$, and let $u :\R^d  \times [0,T] \rightarrow \R$ solve 
    \begin{align}\label{eq:equation_refined_strichartz}
        \begin{dcases}
            i\partial_t u -\frac{1}{\sqrt{\mu}}m_d(\sqrt{\mu}D) u = F, \\
            u(x,0) = f(x),
        \end{dcases}
    \end{align}
    where the symbol $m_d(D)$ is defined in \eqref{eq:m_1}--\eqref{eq:m_2}.
    Then for $d=1$, we have
    \begin{align}\label{Refined:1d}
         \norm{u}_{L_T^2 L_x^\infty}
         &\lesssim
         \mu^{-\frac{1}{5}+\frac{2\theta}{5}} T^{\frac{3}{10}+\frac{2\theta}{5}} \Big(\norm{J^\theta \abs{D}^{-\frac{1}{2}+\frac{3\theta}{2}} u}_{L_T^\infty L_x^2}
         + \norm{J_\mu^{\frac{5}{8}}  u}_{L_T^\infty L_x^2}\Big)
          \\ \notag
          &+ \mu^{-\frac{2}{5}+\frac{\theta}{5}} T^{\frac{3}{5}+\frac{\theta}{5}} \norm{ \abs{D}^{-1} J_\mu^{\frac{5}{8}} F}_{L_T^2 L_x^2}
    \end{align}
    and for $d=2$, it holds
    \begin{align}\label{Refined:2d}
         \norm{u}_{L_T^2 L_x^\infty}
        & \lesssim
         \mu^{-\frac{1}{6} + \frac{5\theta}{36}} T^{\frac{1}{6} + \frac{5\theta}{18}} \Big(\norm{J^{\theta} u}_{L_T^\infty L_x^2}
         +  \norm{\abs{D}^{\frac{1}{2}-\frac{\theta}{3}} J^{2\theta} J_\mu^{\frac{3}{4}}  u}_{L_T^\infty L_x^2}\Big)
          \\ \notag
          &+ \mu^{-\frac{1}{3}+\frac{\theta}{9}} T^{\frac{1}{3}+\frac{2\theta}{9}} \norm{\abs{D}^{-\frac{1}{2}-\frac{\theta}{3}} J^\theta J_\mu^{\frac{3}{4}} F}_{L_T^2 L_x^2}
\end{align}
for $\theta > 0$ arbitrarily small.
\end{cor}
\begin{proof}
    We split the proof into two parts: one for $d=1$ and another for $d=2$.
    \subsubsection*{\underline{Proof of \eqref{Refined:1d}}}
    We start by fixing $q=4^+$ and $r=\infty^-$. More precisely, for $\theta > 0$ arbitrarily small we take $\frac{1}{q} = \frac{1}{4}-\frac{\theta}{4}$ and $\theta=\frac{2}{r}$.
    Then from Lemma \ref{strichartz:lemma} with
\[
  A(\lambda) = A_{\mu,1}(\lambda)
  =
  \mu^{-\frac{1}{2}} \lambda^{-\frac{1}{2}} \jbrack{\sqrt{\mu} \lambda}^{\frac{5}{4}},
\]
the bound \eqref{GenericRefined} now becomes
\begin{align*}
 \norm{u}_{L_T^2 L_x^\infty}
 &\lesssim
 \omega^{\frac{1}{2}-\frac{\theta}{2}} T^{\frac{1}{2}} \norm{J^\theta P_{\le \omega} u}_{L_T^\infty L_x^2}
 \\
 &+ \omega^{-\theta} \mu^{-\frac{1}{4}+\frac{\theta}{4}} T^{\frac{1}{2}} \norm{\rho(\abs{D})^{-\frac{1}{4}+\frac{\theta}{4}} \abs{D}^{-\frac{1}{4}+\frac{\theta}{4}} J^{2\theta} J_\mu^{\frac{5}{8}}  u}_{L_T^\infty L_x^2}
  \\
  &+ \mu^{-\frac{1}{4}+\frac{\theta}{4}} \norm{\rho(\abs{D})^{\frac{3}{4}+\frac{\theta}{4}} \abs{D}^{-\frac{1}{4}+\frac{\theta}{4}} J^\theta J_\mu^{\frac{5}{8}} F}_{L_T^2 L_x^2}.
\end{align*}
We take $\rho(\lambda) = T^{\frac{4}{5}} \mu^{-\frac{1}{5}} \lambda^{-1}$. To get $\rho \le T$, we need $\lambda \ge \mu^{-\frac{1}{5}} T^{-\frac{1}{5}}$, so we choose
\[
  \omega = (\mu T)^{-\frac{1}{5}}.
\]
This yields
\begin{align*}
 \norm{u}_{L_T^2 L_x^\infty}
 &\lesssim
 \mu^{-\frac{1}{10}+\frac{\theta}{10}} T^{\frac{2}{5}+\frac{\theta}{10}} \norm{J^\theta P_{\le \omega} u}_{L_T^\infty L_x^2}
 \\
 &+ \mu^{-\frac{1}{5}+\frac{2\theta}{5}} T^{\frac{3}{10}+\frac{2\theta}{5}} \norm{J^{2\theta} J_\mu^{\frac{5}{8}}  u}_{L_T^\infty L_x^2}
  \\
  &+ \mu^{-\frac{2}{5}+\frac{\theta}{5}} T^{\frac{3}{5}+\frac{\theta}{5}} \norm{J^{\theta} \abs{D}^{-1} J_\mu^{\frac{5}{8}} F}_{L_T^2 L_x^2}.
\end{align*}
To match the powers of $\mu$ and $T$ in the first two terms on the right-hand side, we estimate
\[
   \norm{J^\theta P_{\le \omega} u}_{L_T^\infty L_x^2} = \norm{J^\theta \abs{D}^{\frac{1}{2}-\frac{3\theta}{2}} \abs{D}^{-\frac{1}{2}+\frac{3\theta}{2}} P_{\le \omega} u}_{L_T^\infty L_x^2} \lesssim \omega^{\frac{1}{2}-\frac{3\theta}{2}} \norm{J^\theta \abs{D}^{-\frac{1}{2}+\frac{3\theta}{2}} u}_{L_T^\infty L_x^2} 
\]
so that finally
\begin{align*}
 \norm{u}_{L_T^2 L_x^\infty}
 &\lesssim
 \mu^{-\frac{1}{5}+\frac{2\theta}{5}} T^{\frac{3}{10}+\frac{2\theta}{5}} \norm{J^\theta \abs{D}^{-\frac{1}{2}+\frac{3\theta}{2}} u}_{L_T^\infty L_x^2}
 \\
 &+ \mu^{-\frac{1}{5}+\frac{2\theta}{5}} T^{\frac{3}{10}+\frac{2\theta}{5}} \norm{J^{2\theta} J_\mu^{\frac{5}{8}}  u}_{L_T^\infty L_x^2}
  \\
  &+ \mu^{-\frac{2}{5}+\frac{\theta}{5}} T^{\frac{3}{5}+\frac{\theta}{5}} \norm{ J^{\theta}\abs{D}^{-1} J_\mu^{\frac{5}{8}} F}_{L_T^2 L_x^2}
\end{align*}
for $\theta > 0$ sufficiently small.
\subsubsection*{\underline{Proof of \eqref{Refined:2d}}}
We choose $q= 2^+$ and $r= \infty^-$. So, for $\theta > 0$ arbitrarily small, we take $\frac{1}{q} = \frac{1}{2}-\frac{\theta}{3}$ and $\theta=\frac{3}{r}$. We also let $\frac{1}{\Tilde{r}}=\frac{5\theta }{12}$ and $\Tilde{\theta} =\theta$. Then from Lemma \ref{strichartz:lemma}, we have
\[
  A(\lambda) = A_{\mu,2}(\lambda)
  =
  \mu^{-\frac{1}{2}} \jbrack{\sqrt{\mu} \lambda}^{\frac{3}{2}}.
\]
Consequently, the bound \eqref{GenericRefined} gives us
\begin{align*}
 \norm{u}_{L_T^2 L_x^\infty}
 &\lesssim
 \omega^{1-\frac{5\theta}{6}} T^{\frac{1}{2}} \norm{J^{\theta} P_{\le \omega} u}_{L_T^\infty L_x^2}
 \\
 &+ \omega^{-\frac{\theta}{6}} \mu^{-\frac{1}{4}+\frac{\theta}{6}} T^{\frac{1}{2}} \norm{\rho(\abs{D})^{-\frac{1}{2}+\frac{\theta}{3}} J^{2\theta} J_\mu^{\frac{3}{4}}  u}_{L_T^\infty L_x^2}
  \\
 & + \mu^{-\frac{1}{4}+\frac{\theta}{6}} \norm{\rho(\abs{D})^{\frac{1}{2}+\frac{\theta}{3}} J^\theta J_\mu^{\frac{3}{4}} F}_{L_T^2 L_x^2}.
\end{align*}
We let $\rho(\lambda) = T^{\frac{2}{3}} \mu^{-\frac{1}{6}} \lambda^{-1}$. Then it requires $\lambda \ge \mu^{-\frac{1}{6}} T^{-\frac{1}{3}}$ to maintain $\rho \le T$, so we take
\[
  \omega = \mu^{-\frac{1}{6}} T^{-\frac{1}{3}}.
\]
This gives us
\begin{align*}
 \norm{u}_{L_T^2 L_x^\infty}
& \lesssim
 \mu^{-\frac{1}{6} + \frac{5\theta}{36}} T^{\frac{1}{6} + \frac{5\theta}{18}} \norm{J^{\theta} u}_{L_T^\infty L_x^2}
 \\
 &+ \mu^{-\frac{1}{6}+\frac{5\theta}{36}} T^{\frac{1}{6}+\frac{5\theta}{18}} \norm{\abs{D}^{\frac{1}{2}-\frac{\theta}{3}} J^{2\theta} J_\mu^{\frac{3}{4}}  u}_{L_T^\infty L_x^2}
  \\
  &+ \mu^{-\frac{1}{3}+\frac{\theta}{9}} T^{\frac{1}{3}+\frac{2\theta}{9}} \norm{\abs{D}^{-\frac{1}{2}-\frac{\theta}{3}} J^\theta J_\mu^{\frac{3}{4}} F}_{L_T^2 L_x^2}
\end{align*}
for $\theta > 0$ small enough.

\end{proof}

\section{The Whitham equation} \label{sec:Whitham}
In this section we prove Theorem \ref{thm:wellposedness_Whitham}, 
where we wish to extend the lifespan from $1/\epsilon$, achieved in \cite{Klein2018}, to
\begin{align}\label{eq:w_time_1}
    T = c^{-1} \epsilon^{-1} \Big(\frac{\mu}{\epsilon}\Big)^{\frac{1}{4}^-}
\end{align}
for some $c= c(\norm{\eta_0}_{H^s_x})>0$. This will be enough to prove Theorem \ref{thm:wellposedness_Whitham}, since the uniqueness, continuous dependence on the data, and the persistence at higher regularity, on the other hand, are local-in-time properties and follows form the existing local well-posedness result (stated in Proposition \ref{prop:existence_whitham} below); see Remarks \ref{rem:w_rem1}--\ref{rem:w_rem3} for more details. 

To prove that the solution exists on the timescale \eqref{eq:w_time_1}, we need the standard energy estimate and a refined Strichartz estimate deduced in Sections \ref{sec:energy_whitham}--\ref{sec:refined_str_whitham}, respectively.

Throughout this paper, we will need the classical Kato-Ponce estimates.
\begin{prop}[Kato-Ponce estimates \cite{Kato1988}]
    Let $s \geq 0$, $p\in (1,\infty)$, $f\in H^s(\R^d)$, and $g\in H^{s-1}(\R^d)$.  Then
    \begin{align}\label{eq:kato_ponce}
        \norm{[J^s,f] g}_{L^p_x} \lesssim \norm{\nabla f}_{L^\infty_x} \norm{J^{s-1} g}_{L^p_x} + \norm{J^s f}_{L^p_x} \norm{ g}_{L^\infty_x}.
    \end{align}
    Moreover,  let $p_1,p_4 \in (1,\infty]$ and $p_2,p_3\in (1,\infty)$ be such that $\frac{1}{p}= \frac{1}{p_1} + \frac{1}{p_2} = \frac{1}{p_3} + \frac{1}{p_4}$. Then
    \begin{align}\label{eq:kato_ponce_ineq}
        \norm{J^s(f g)}_{L^p_x} \lesssim \norm{f}_{L^{p_1}_x} \norm{J^{s} g}_{L^{p_2}_x}  + \norm{J^s f}_{L^{p_3}_x} \norm{ g}_{L^{p_4}_x}
    \end{align}
    for $f,g\in H^s(\R^d)$.
\end{prop}

\subsection{Energy estimate for the Whitham equation}\label{sec:energy_whitham} 
We define the energy for the Whitham equation \eqref{eq:Whitham} to be
\begin{align*}
    \mathcal{E}_s(\eta) = \int_\R \Big(J^s \eta \Big)^2 \d x,
\end{align*}
which allows us to deduce the following estimate, under some additional regularity.

\begin{prop}\label{prop:energy1_whitham}
    Let $s>\frac{3}{2}$ and $\epsilon,\mu \in (0,1]$. Let $r \ge s+1$ and $T > 0$. Then for any solution $\eta \in C([0,T];H^r(\R))$ to \eqref{eq:Whitham} on the time interval $[0,T]$ we have
     \begin{align}\label{eq:energy_est_whitham}
         \frac{\d}{\d t} (\norm{\eta(t)}_{H^s_x}^2) \le c  \epsilon \norm{\partial_x \eta(t)}_{L^\infty_x} \norm{\eta(t)}_{H^s_x}^2
     \end{align}
     for $t\in(0,T)$.
\end{prop}
\begin{proof}
    From \eqref{eq:Whitham}, we have 
    \begin{align*}
         \frac{1}{2}\frac{\d}{\d t} (\norm{\eta}_{H^s_x}^2) &= \scalarprod{J^s \eta}{J^s \partial_t \eta}_{L^2_x} \\
         &= - \scalarprod{J^s \eta}{J^s \sqrt{\mathcal{T}_\mu}(D) \partial_x \eta}_{L^2_x} - \epsilon \scalarprod{J^s \eta}{J^s (\eta\partial_x \eta)}_{L^2_x}.
    \end{align*}
    Here, the first term is skew-symmetric, and the second term can be estimated by H\"{o}lder's inequality, the Kato-Ponce commutator estimate \eqref{eq:kato_ponce}, and integration by parts. So, we obtain
    \begin{align*}
       \frac{\d}{\d t} (\norm{\eta}_{H^s_x}^2) 
       &\leq 2\epsilon |\scalarprod{J^s \eta}{[J^s ,\eta]\partial_x \eta}_{L^2_x}| + 2\epsilon |\scalarprod{J^s \eta}{\eta J^s \partial_x \eta}_{L^2_x}|\\
        &\le c \epsilon \norm{\partial_x \eta}_{L^\infty_x} \norm{\eta}^2_{H^s_x}
    \end{align*}
    as desired.
\end{proof}

Via Gr\"onwall's lemma and approximation we then obtain the following bound.
\begin{cor}\label{cor:energy1_whitham}
    Let $s>\frac{3}{2}$, $T > 0$ and $\epsilon,\mu \in (0,1]$. Then for any solution $\eta \in C([0,T];H^s(\R))$ to \eqref{eq:Whitham} on $[0,T]$ with $\eta(0)=\eta_0$ we have
     \begin{equation}\label{eq:gronwall_whitham}
         \norm{\eta(t)}_{H^s_x}^2
         \le
         \mathrm{e}^{c \epsilon \int_0^t \norm{\partial_x \eta(\tau)}_{L^\infty_x} \!\d \tau} \norm{\eta_0}_{H^s_x}^2
     \end{equation}
     for $t\in [0,T]$.
\end{cor}

\begin{proof}
First we notice that it suffices to prove \eqref{eq:gronwall_whitham} for an arbitrarily small $T > 0$, since this result can then be iterated to cover any time interval. By the local existence theorem, Proposition \ref{prop:existence_whitham} below, there exists a time $T > 0$ such that the following holds. Taking $r=s+1$ and choosing an approximating sequence $\eta_{0}^{(n)} \in H^r(\R)$ such that $\eta_{0}^{(n)} \to \eta_0$ in $H^s(\R)$ as $n \to \infty$, there exist corresponding solutions $\eta^{(n)}\in C([0,T];H^r(\R))$, and by the continuous dependence on the data we have that $\eta^{(n)} \to \eta$ in $C([0,T];H^s(\R))$. From Proposition \ref{prop:energy1_whitham} it follows that \eqref{eq:gronwall_whitham} holds with $\eta$ replaced by $\eta^{(n)}$, and passing to the limit $n \to \infty$ we get the inequality also for $\eta$.
\end{proof}

\subsection{Refined Strichartz estimate for the Whitham equation}\label{sec:refined_str_whitham}
To obtain the desired lifespan, we use Corollary \ref{cor:refined_strichartz_d_1} to deduce the following estimate. 
\begin{prop}\label{prop:boundedness_whitham}
    Let $\epsilon,\mu\in (0,1]$, $r>\frac{13}{8}$, and $T>0$. Let $\eta \in C([0,T];H^{r}(\R))$ be a solution to \eqref{eq:Whitham}. Then
    \begin{align}\label{eq:boundedness_whitham}
      \norm{\partial_x \eta}_{L^2_TL^\infty_x} &\lesssim  \mu^{(-\frac{1}{5})^+ }   T^{\frac{3}{10}^+}  \norm{\eta}_{L^\infty_TH^{r}_x} +  \epsilon  \mu^{(-\frac{2}{5})^+}  T^{\frac{11}{10}^+} \norm{\eta}_{L^\infty_T H^{r}_x}^2.
      \end{align}
\end{prop}
\begin{proof}
    We observe that the Whitham equation \eqref{eq:Whitham} has the same form as \eqref{eq:equation_refined_strichartz} for $u=\eta$ and $F = -i\frac{\epsilon}{2}\partial_x( \eta^2)$. With these choices, we have from \eqref{Refined:1d} and H\"{o}lder's inequality in time that
    \begin{align*}
        \norm{\partial_x \eta}_{L_T^2 L_x^\infty}
         &\lesssim
         \mu^{-\frac{1}{5}+\frac{2\theta}{5}} T^{\frac{3}{10}+\frac{2\theta}{5}} \Big(\norm{J^\theta \abs{D}^{\frac{1}{2}+\frac{3\theta}{2}} \eta}_{L_T^\infty L_x^2} +  \norm{J_\mu^{\frac{5}{8}}  |D| \eta}_{L_T^\infty L_x^2}\Big)
          \\ \notag
          &+ \epsilon \mu^{-\frac{2}{5}+\frac{\theta}{5}} T^{\frac{11}{10}+\frac{\theta}{5}} \norm{ J_\mu^{\frac{5}{8}} |D|  (\eta^2)}_{L_T^\infty L_x^2}
    \end{align*}
    for sufficiently small $\theta >0$.  Then by the Kato-Ponce inequality \eqref{eq:kato_ponce_ineq} and a Sobolev embedding we obtain \eqref{eq:boundedness_whitham} 
    for $r>\frac{13}{8}$. 
\end{proof}

\subsection{Proof of the Theorem \ref{thm:wellposedness_Whitham}}\label{sec:proof_of_Whitham} 
We will now combine the estimates in Corollary \ref{cor:energy1_whitham} and Proposition \ref{prop:boundedness_whitham} with a bootstrap argument to prove Theorem \ref{thm:wellposedness_Whitham}.

\begin{proof}[Proof of the Theorem \ref{thm:wellposedness_Whitham}]
	Arguing as in \cite{Saut1979} and \cite{Bona1975}, one get the following local well-posedness result for the Whitham equation \eqref{eq:Whitham}:
    \begin{prop}\label{prop:existence_whitham}
        Let $s>\frac{3}{2}$, $\epsilon,\mu \in (0,1]$, and $\eta_0 \in H^{s}(\R)$. Then there exists a positive time $T = T(\norm{\eta_0}_{H^s_x})$, decreasing with respect to its arguement, such that there is a unique $\eta \in C([0, T ] ; H^s(\R))$ solving \eqref{eq:Whitham} with initial condition $\eta(0) = \eta_0$. Moreover,
        \begin{align*}
            \sup_{t\in [0,T]} \norm{\eta(t)}_{H^s_x} \lesssim\norm{\eta_0}_{H^s_x},
        \end{align*}
        and the flow map is continuously dependent on the initial datum.  Furthermore, higher regularity persists. That is, if $\eta_0 \in H^r(\R)$ for some $r > s$, then $\eta$ belongs to the space $C([0, T ] ; H^r(\R))$ (with the same $T$).
    \end{prop}
    \begin{remark}
        The time of existence of order $1/\epsilon$ was obtained in \cite{Klein2018}.
    \end{remark}

We only need to prove that the time of existence is of order $\epsilon^{-1} (\mu/\epsilon)^{\frac{1}{4}^-}$. 
 Indeed, the uniqueness, continuous dependence on the data, and persistence of higher regularity are local-in-time properties and follow directly from Proposition \ref{prop:existence_whitham}; see Remarks \ref{rem:w_rem1}--\ref{rem:w_rem3} below.


We notice that a consequence of Proposition \ref{prop:existence_whitham}, the initial datum $\eta_0 \in H^s(\R)$, for $s> \frac{3}{2}$, gives the rise to a solution to \eqref{eq:Whitham}, $\eta \in C([0,T^\ast);H^s(\R))$, where $T^\ast \geq T(\norm{\eta_0}_{H^s_x})$ is the maximal time of existence. Here, $T$ is a decreasing function of its argument. Then it also follows that the solution satisfies the blow-up alternative:
    \begin{align}\label{eq:BA}
        \text{If } \ \ T^\ast <\infty, \text{ then } \limsup_{t\nearrow T^\ast} \norm{\eta(t)}_{H^s_x} = \infty
    \end{align}
    (see, e.g., \cite[Section 5.2]{Iorio2001}).


With Proposition \ref{prop:existence_whitham}, we may now obtain the a priori estimates and prove that the solution exists on a larger timescale. This is summarized in the lemma below.

\begin{lem}\label{lem:bootstrap_whitham}
Let $s>\frac{13}{9}$ and $\epsilon,\mu \in (0,1]$. Then there exists a positive time 
\begin{align}\label{eq:w_time_t0}
     T_0 = c^{-1} \epsilon^{-1} \left( \frac{\mu}{\epsilon} \right)^{\frac{1}{4}^-} \norm{\eta_0}_{H^s_x}^{(-\frac{5}{4})^+}
\end{align}
such that $T^\ast > T_0$ and 
\begin{align*}
    \sup_{t\in [0,T_0]}\norm{\eta(t)}_{H^s_x}  \leq 4 \norm{\eta_0}_{H^s_x} .
\end{align*}
\end{lem}


\begin{proof}
    We will argue by contradiction by defining 
    \begin{align}\label{eq:T_tilde_nu_whitham}
        \Tilde{T} := \sup \{ T \in (0,T^\ast) : \sup_{t\in [0,T]} \norm{\eta(t)}_{H^{s}_x} \leq 4 \norm{\eta_0}_{H^{s}_x} \}
    \end{align}
    and assuming that $\Tilde{T} \leq T_0$. Obtaining an estimate inside the ball in \eqref{eq:T_tilde_nu_whitham} for the time $T_0$ defined in \eqref{eq:w_time_t0} will lead to a contradiction.

    Fix $T_1 \in (0,\Tilde{T})$, then we utilize \eqref{eq:gronwall_whitham},  \eqref{eq:boundedness_whitham}, and the assumption $T_1 <\Tilde{T} \leq T_0$ to obtain
    \begin{align*}
        \norm{\eta(t)}_{H^{s}_x}^2 &\leq  \mathrm{e}^{c\epsilon T_1^\frac{1}{2}  \norm{\partial_x \eta}_{L^2_{T_1}L^\infty_x} }   \norm{\eta_0}_{H^{s}_x}^2\\
        &\leq \mathrm{e}^{c \mu^{(-\frac{1}{5})^+} \epsilon T_0^{\frac{4}{5}^+} \norm{\eta_0}_{H^{s}_x} + c \mu^{(-\frac{2}{5})^+} \epsilon^2 T_0^{\frac{8}{5}^+} \norm{\eta_0}_{H^{s}_x}^2 }   \norm{\eta_0}_{H^{s}_x}^2\\
        &=: \mathrm{e}^{c(\delta + \delta^2)}   \norm{\eta_0}_{H^{s}_x}^2
    \end{align*}
    for all $t\in [0,T_1]$ and $s>\frac{13}{8}$. Choosing $T_0>0$ small enough such that
    \begin{align*}
        2c \delta < \log 2,
    \end{align*}
    yielding $c(\delta + \delta^2) < 2c \delta $ and \eqref{eq:w_time_t0},  implies
    \begin{align*}
        \norm{\eta(t)}_{H^{s}_x} & \leq 2\norm{\eta_0}_{H^{s}_x}
    \end{align*}
    for all $t\in [0,T_1]$. As a consequence of \eqref{eq:BA}, we have that $T_1< T^\ast$. Then by the continuity of the solution $\eta(t)$ for $t\in [0,T^\ast)$ from the definition of $T^\ast$, there exists a time $\tau \in [\Tilde{T},T^\ast)$ such that 
    \begin{align*}
        \norm{\eta(\tau)}_{H^s_x} \leq 3\norm{\eta_0}_{H^s_x}.
    \end{align*}
    But this leads to a contradiction of the definition of $\Tilde{T}$ in \eqref{eq:T_tilde_nu_whitham}. We may therefore conclude that indeed $T_0 < \Tilde{T}$, where $T_0$ is defined in \eqref{eq:w_time_t0}.
    \renewcommand\qedsymbol{$\overset{\text{Lemma \ref{lem:bootstrap_whitham}}}{\ensuremath{\Box}}$}
\end{proof}

\SkipTocEntry\subsubsection*{Conclusion.}  
We conclude the proof of Theorem \ref{thm:wellposedness_Whitham} with some remarks on uniqueness, continuous dependence on data, and persistence of higher regularity. All these properties are known for some positive time by the local existence and uniqueness result, Proposition \ref{prop:existence_whitham}, and they therefore extend to the larger time interval $[0,T_0]$ by standard arguments which we briefly recall in the following remarks. In these remarks we assume that $\eta,\tilde\eta \in C([0,T];H^s(\R))$ are two solutions of the Whitham equation with initial data $\eta_0,\tilde\eta_0\in H^s(\R)$, respectively. Here $T > 0$ is completely arbitrary.

\begin{remark}\label{rem:w_rem1}
For uniqueness, assume that $\eta_0=\tilde\eta_0$. If uniqueness fails, then for some $S \in (0,T)$ we have $\eta=\tilde\eta$ in $[0,S]$ while this fails on $[S,S+\delta]$ for all $\delta > 0$. But applying Proposition \ref{prop:existence_whitham} with initial time $t=S$, this contradicts the uniqueness part of the proposition.
\end{remark}

\begin{remark}\label{rem:w_rem2}
For continuous dependence on the data, we must show that given any $\varepsilon > 0$ there exists $\delta > 0$ such that if $\norm{\eta_0-\tilde\eta_0}_{H^s_x} < \delta$, then $\sup_{t \in [0,T]} \norm{\eta(t)-\tilde\eta(t)}_{H^s_x} < \varepsilon$. Choose $R > 0$ so large that $\norm{\eta(t)}_{H^s_x}, \norm{\tilde\eta(t)}_{H^s_x} \le R$ for all $t \in [0,T]$. Now let $S=S(R) > 0$ be the time in Proposition \ref{prop:existence_whitham}. We may assume that $S=T/N$ for a large integer $N$. Then applying the continuous dependence on the data in Proposition \ref{prop:existence_whitham} first at the initial time $t_1:=T-S$, we find first $\delta_1 > 0$ such that if $\norm{\eta(t_1)-\tilde\eta(t_1)}_{H^s_x} < \delta_1$, then $\sup_{t \in [t_1,T]} \norm{\eta(t)-\tilde\eta(t)}_{H^s_x} < \varepsilon$. Next, repeating the argument at the initial time $t_2:=t_1-S=T-2S$ we obtain $\delta_2 \in (0,\delta_1]$ such that if $\norm{\eta(t_2)-\tilde\eta(t_2)}_{H^s_x} < \delta_2$, then $\sup_{t \in [t_2,t_1]} \norm{\eta(t)-\tilde\eta(t)}_{H^s_x} < \delta_1$, etc.
\end{remark}

\begin{remark}\label{rem:w_rem3}
For persistence of higher regularity, assume that $\eta_0$ belongs to $H^r(\R)$ for some $r > s$. Choose $R > 0$ so large that $\norm{\eta(t)}_{H^s_x} \le R$ for all $t \in [0,T]$, let $S=S(R) > 0$ be the time in Proposition \ref{prop:existence_whitham}, and assume that $S=T/N$ for an integer $N$. Applying the persistence property in Proposition \ref{prop:existence_whitham} on the successive time intervals $[0,S]$, $[S,2S]$ etc., we conclude that $\eta$ belongs to  $C([0,T];H^r(\R))$.
\end{remark}

\end{proof}

\section{The Whitham-Boussinesq system}\label{sec:WB_d_2}
In this section, we will give the proof of Theorem \ref{thm:wellposedness_WB}. For brevity, we will only give the proof in the two-dimensional case. However, the proof for dimension one follows the same procedure. 

To prove Theorem \ref{thm:wellposedness_WB}, we only need to prove that the solution exists on the improved lifespan $\epsilon^{-\frac{d+3}{4}} (\mu/\epsilon)^{\frac{1}{4}^-}$. To be precise, the uniqueness, the continuous dependence on the initial data, and the persistence at higher regularity will be a consequence of the existing local well-posedness result (stated in Theorem \ref{thm:paulsen} below). The improved lifespan will be obtained by combining the energy method with a refined Strichartz estimate,
where we will first need to prove two commutator estimates and state some results needed to prove Theorem \ref{thm:wellposedness_WB} (see Sections \ref{sec:commutator_ests}--\ref{sec:refined_str_wb}). This will all come together in Section \ref{sec:proof_main_res_two_dims}  to prove Theorem \ref{thm:wellposedness_WB}.
Throughout this section, we let $x=(x_1,x_2)\in \R^2$ and $\mathbf{v} = (v_1,v_2)^\mathrm{T}$.

\subsection{Commutator and pointwise estimates}\label{sec:commutator_ests}
We will need to derive commutator estimates for the Fourier multiplier $\mathcal{T}_\mu(D)$ and its inverse. These commutator estimates will be deduced from a generalization of the Kato-Ponce commutator estimate \eqref{eq:kato_ponce}, which holds for the class of symbols defined below. 
\begin{definition}[Symbol class {\cite[Definition B.7]{Lannes2013}}] \label{def:symb_class}
    We say that a symbol $\sigma(D)$ is an element of the symbol class $\mathtt{S}^s$ with $s\in \R$, if $\xi \in \R^d \mapsto \sigma (\xi) \in \mathbb{C}$ is smooth and satisfies 
    \begin{align*}
        \forall \alpha \in \N^d, \ \ \sup_{\xi \in \R^d} \ \jbrack{\xi}^{|\alpha|- s} \Big| \frac{\partial^\alpha}{\partial\xi^\alpha} \sigma(\xi)  \Big| < \infty. 
    \end{align*}
    We also introduce the following semi-norm
    \begin{align*}
        \mathcal{N}^s(\sigma) = \sup_{\alpha\in \N^d, \:  |\alpha| \leq 2+d+\lceil\frac{d}{2} \rceil}\: \sup_{\xi \in \R^d} \ \jbrack{\xi}^{|\alpha|-s} \Big| \frac{\partial^\alpha}{\partial \xi^\alpha }  \sigma(\xi) \Big|.
    \end{align*}
\end{definition}
This symbol class satisfies the following commutator estimate found in \cite[Proposition B.8]{Lannes2013} (see also \cite{Lannes2006}).
\begin{lem}\label{lem:gen_kato_ponce}
    Let $s\geq 0$, $t_0>\frac{d}{2}$, and $\sigma \in \mathtt{S}^s$. If $f\in H^s(\R^d)\cap H^{t_0+1}(\R^d)$ and $g\in H^{s-1}(\R^d) \cap H^{t_0}(\R^d)$, then
    \begin{align}\label{eq:gen_kato_ponce_Linfty}
        \norm{[\sigma(D),f]g}_{L^2_x} &\lesssim \mathcal{N}^s(\sigma) (\norm{\nabla f}_{L^\infty_x}\norm{g}_{H^{s-1}_x} + \norm{\nabla f}_{H^{s-1}_x}\norm{g}_{L^\infty_x}).
\end{align}
\end{lem}
To use Lemma \ref{lem:gen_kato_ponce} when proving the desired commutator estimates, we need the following pointwise results on the symbol $\mathcal{T}_\mu(\xi)$ that are found in \cite[Lemma 2.5 and A.1]{Paulsen2022} for $d=1$. 
\begin{lem}\label{lem:derivatives_Tilbert}
    Let $\mu \in (0,1]$. Then the symbol $\mathcal{T}_\mu(\xi)$ defined in \eqref{eq:def_T_mu_xi} satisfies
    \begin{align}
    \notag
        \mathcal{T}_\mu (\xi) \jbrack{\sqrt{\mu}\xi} &\sim 1,
        \\
        \Big|\frac{\partial^\alpha}{\partial \xi^\alpha}\sqrt{\mathcal{T}_\mu (\xi)} \Big| &\lesssim \mu^{\frac{|\alpha|}{2}} \jbrack{\sqrt{\mu}\xi}^{-\frac{1}{2} - |\alpha|}
        \label{eq:derivatives_Tilbert}
    \end{align}
    for any $\alpha\in \N^d$.
\end{lem}  
The next result is a direct consequence of the chain rule and Lemma \ref{lem:derivatives_Tilbert}.
\begin{cor}\label{cor:est_on_Tilbert_symbol}
Let $\mu \in (0,1]$. Then the symbol $\sqrt{\mathcal{T}_\mu(\xi)}$ defined in \eqref{eq:def_T_mu_xi} satisfies
    \begin{align}\label{eq:derivatives_Tilbert_inv}
        \Big|\frac{\partial^\alpha}{\partial \xi^\alpha}\frac{1}{\sqrt{\mathcal{T}_\mu (\xi)}} \Big| & \lesssim \mu^{\frac{|\alpha|}{2}} \jbrack{\sqrt{\mu}\xi}^{\frac{1}{2}- |\alpha|}
    \end{align}
    for any $\alpha\in \N^d$.
\end{cor}

From Lemmas \ref{lem:gen_kato_ponce}--\ref{lem:derivatives_Tilbert} and Corollary \ref{cor:est_on_Tilbert_symbol}, we are able to derive two commutator estimates that will come in handy later.

\begin{lem}\label{lem:commutator_est_tilbert}
    Let  $\mu \in (0,1]$, $s\geq 0$, $t_0>\frac{d}{2}$, $f\in H^s(\R^d)\cap H^{t_0+1}(\R^d)$, and $g\in H^{s-1}(\R^d) \cap H^{t_0}(\R^d)$, then
    \begin{align} \label{eq:commutator_est_tilbert2}
        \norm{[J^s\sqrt{\mathcal{T}_\mu}(D),f] g}_{L^2_x} &\lesssim \norm{\nabla f}_{L^\infty_x}\norm{J^{s-1} g}_{L^2_x} + \norm{J^s f}_{L^2_x}\norm{g}_{L^\infty_x}.
    \end{align}
\end{lem}
\begin{remark}
    The calculations for the verification of $J^s\sqrt{\mathcal{T}_\mu}(D) \in \mathtt{S}^s$ and $\mathcal{N}^s(J^s \sqrt{\mathcal{T}_\mu}) \lesssim 1$ uniformly in $\mu\in(0,1]$ are similar to those found in \cite[p. 6344]{Paulsen2022} for $s=0$. However, since we require sharper commutator estimates than those presented in \cite{Paulsen2022}, we include the calculations for the convenience of the reader. In particular, we need energy estimates with one part in $L^\infty(\R^d)$ to utilize the refined Strichartz estimates (see Corollary \ref{cor:refined_strichartz_d_1} and Proposition \ref{prop:boundedness_eta_v_d_2}). This will be a key part of achieving an enhanced lifespan.
\end{remark}
\begin{proof}
    We will use \eqref{eq:gen_kato_ponce_Linfty} to obtain the commutator estimate \eqref{eq:commutator_est_tilbert2}. In other words, we need to verify that  $J^s \sqrt{\mathcal{T}_\mu} \in \mathtt{S}^s$ and calculate $\mathcal{N}^s(J^s \sqrt{\mathcal{T}_\mu})$, which are defined in Definition \ref{def:symb_class}. From  the Leibniz rule and \eqref{eq:derivatives_Tilbert}, it follows that
    \begin{align*}
        \jbrack{\xi}^{|\alpha|-s} \Big| \frac{\partial^\alpha}{\partial\xi^\alpha} \Big( \sqrt{\mathcal{T}_\mu(\xi)} \jbrack{\xi}^{s}\Big) \Big|  \lesssim \jbrack{\xi}^{|\alpha|-s} \sum_{\gamma\in \N^d\: : \: \gamma\leq \alpha} \mu^{\frac{|\alpha|-|\gamma|}{2}} \jbrack{\sqrt{\mu}\xi}^{-\frac{1}{2} - (|\alpha|-|\gamma|)} \jbrack{\xi}^{s-|\gamma|} \lesssim  1,
    \end{align*}
    where the implicit constant is independent of $\mu \in (0,1]$. Consequently, $J^s\sqrt{\mathcal{T}_\mu} \in \mathtt{S}^s$ and $\mathcal{N}^s(J^s\sqrt{\mathcal{T}_\mu} ) \lesssim 1$ uniformly in $\mu$. Then, we have from \eqref{eq:gen_kato_ponce_Linfty} that 
    \begin{align*}
        \norm{[J^s \sqrt{\mathcal{T}_\mu}(D),f] g}_{L^2_x} 
        &\lesssim \norm{\nabla f}_{L^\infty_x}\norm{g}_{H^{s-1}_x} + \norm{f}_{H^s_x}\norm{g}_{L^\infty_x}
    \end{align*}
    as desired. 
\end{proof}

\begin{lem}\label{lem:commutator_inv_tilbert}
    Let $\mu\in(0,1]$, $s\geq 0$, $t_0 >\frac{d}{2}$, $f\in H^s(\R^d)\cap H^{t_0+1}(\R^d)$, and $g\in H^{s-1}(\R^d) \cap H^{t_0}(\R^d)$. Then 
    \begin{align}\label{eq:commutator_est_tilbert_inv}
        \norm{[ J^s \mathcal{T}_{\mu}^{-\frac{1}{2}}(D) , f]g}_{L^2_x} &\lesssim \norm{\nabla f}_{L^\infty_x} \norm{J^{s-1} J^\frac{1}{2}_\mu g}_{L^2_x} + \norm{J^s J^\frac{1}{2}_\mu f}_{L^2_x}\norm{g}_{L^\infty_x}
        .
    \end{align}
\end{lem}

\begin{proof}
Let $\chi^{(j)} \in \mathcal{S}(\R^d)$  be Fourier multipliers for $j=1,2$, that is,
    \begin{align*}
        \mathcal{F}(\chi^{(j)}(D)f)(\xi) = \chi^{(j)}(|\xi|) \widehat{f}(\xi)
    \end{align*}
    for all $f\in \mathcal{S}(\R^d)$ such that
    \begin{align*}
        0\leq \chi^{(j)}(|\xi|) \leq 1, \ \ \ \chi^{(1)}(|\xi|) + \chi^{(2)}(|\xi|) = 1
    \end{align*}
    on $\R^d$ with the support 
    \begin{align*}
        \mathrm{supp} \: \chi^{(1)}(|\xi|) \subset [-2,2], \ \ \ \mathrm{supp} \: \chi^{(2)}(|\xi|) \subset \R\backslash[-1,1].
    \end{align*}
    We let
    \begin{align*}
        \chi^{(j)}_\mu(|\xi|) := \chi^{(j)}(\sqrt{\mu}|\xi|)
    \end{align*}
    be the scaled version.
  With these smooth cut-off functions, we can write
    \begin{align*}
        \mathcal{T}_{\mu}^{-\frac{1}{2}}(D) &=(\chi^{(1)}_\mu \mathcal{T}_{\mu}^{-\frac{1}{2}})(D) + (\chi^{(2)}_\mu \mathcal{T}_{\mu}^{-\frac{1}{2}})(D) .
    \end{align*}
     It then follows that
     \begin{align}\label{eq:commutator_part_123}
         \norm{[ J^s \mathcal{T}_{\mu}^{-\frac{1}{2}}(D) , f]g }_{L^2_x} 
         &\leq \norm{[J^s (\chi^{(1)}_\mu \mathcal{T}_{\mu}^{-\frac{1}{2}})(D) , f]g}_{L^2_x} 
         +\norm{[J^s (\chi^{(2)}_\mu \mathcal{T}_{\mu}^{-\frac{1}{2}})(D) , f]g}_{L^2_x} 
         \\ \notag
         &=: \mathscr{A}_1 + \mathscr{A}_2.
     \end{align}
     To obtain \eqref{eq:commutator_est_tilbert_inv}, we will estimate each of these terms separately, utilizing \eqref{eq:gen_kato_ponce_Linfty}.
     \SkipTocEntry\subsubsection*{Control of $\mathscr{A}_1$ in \eqref{eq:commutator_part_123}}
     We start by estimating $\mathscr{A}_1$, i.e., verifying that $J^s(\chi^{(1)}_\mu \mathcal{T}_{\mu}^{-\frac{1}{2}})\in \mathtt{S}^{s}$ and $\mathcal{N}^s(J^s(\chi^{(1)}_\mu \mathcal{T}_{\mu}^{-\frac{1}{2}})) \lesssim 1$ uniformly in $\mu \in (0,1]$. We first make the observation that 
    \begin{align}\label{eq:chi_est}
        \mu^\frac{|\gamma|}{2}\jbrack{\xi}^{|\gamma|} \Big|\Big( \frac{\partial^\gamma}{\partial \xi^\gamma}\chi_\mu^{(1)}\Big)(\sqrt{\mu}|\xi|) \Big| \lesssim 1
    \end{align}
    for  $\gamma \in \N^d$. Additionally, on the support of $\chi^{(1)}_\mu$ and its derivatives, the estimate \eqref{eq:derivatives_Tilbert_inv} becomes
    \begin{align*}
        \Big| \frac{\partial^\gamma}{\partial\xi^\gamma}\frac{1}{\sqrt{\mathcal{T}_\mu(\xi)}}  \Big| & \lesssim \mu^{\frac{|\gamma|}{2}} \jbrack{\sqrt{\mu}\xi}^{\frac{1}{2} - |\gamma|} \lesssim \jbrack{\xi}^{-|\gamma|}.
    \end{align*}
    Then, the Leibniz rule gives us
    \begin{align*}
        \Big| \frac{\partial^\gamma}{\partial\xi^\gamma} \Big(\chi^{(1)}_\mu(|\xi|)\frac{1}{\sqrt{\mathcal{T}_\mu(\xi)}} \Big) \Big|
        &\lesssim \sum_{\kappa \in \N^d \: :  \: \kappa \leq \gamma} \mu^\frac{|\gamma|-|\kappa|}{2} \Big|\Big( \frac{\partial^{\gamma-\kappa}}{\partial \xi^{\gamma-\kappa}}\chi^{(1)}\Big)(\sqrt{\mu}|\xi|) \Big| \jbrack{\xi}^{-|\kappa|}.
    \end{align*}
    With the above estimate, the Leibniz rule, and \eqref{eq:chi_est}, we obtain
    \begin{align*}
        \jbrack{\xi}^{|\alpha|-s} \Big| \frac{\partial^\alpha}{\partial\xi^\alpha} \Big(\jbrack{\xi}^s \chi^{(1)}_\mu\frac{1}{\sqrt{\mathcal{T}_\mu(\xi)}}\Big)  \Big|
        &\lesssim \sum_{\gamma\leq \alpha}    \sum_{\kappa\leq \gamma} \mu^\frac{|\gamma|-|\kappa|}{2} \Big|\Big( \frac{\partial^{\gamma-\kappa}}{\partial \xi^{\gamma-\kappa}}\chi^{(1)}\Big)(\sqrt{\mu}|\xi|) \Big| \jbrack{\xi}^{|\gamma|-|\kappa|}  \lesssim 1.
    \end{align*}
    Then \eqref{eq:gen_kato_ponce_Linfty} yields 
    \begin{align}\label{eq:B1}
        \mathscr{A}_1 &\lesssim \norm{\nabla f}_{L^\infty_x}\norm{J^{s-1}  g}_{L^2_x} + \norm{J^s f}_{L^2_x}\norm{g}_{L^\infty_x}\\ \notag
        &\lesssim  \norm{\nabla f}_{L^\infty_x} \norm{J^{s-1} J^\frac{1}{2}_\mu g}_{L^2_x} + \norm{J^s J^\frac{1}{2}_\mu f}_{L^2_x}\norm{ g}_{L^\infty_x}.
    \end{align} 
    \SkipTocEntry\subsubsection*{Control of $\mathscr{A}_2$ in \eqref{eq:commutator_part_123}}
    The term $\mathscr{A}_2$ in \eqref{eq:commutator_part_123} will also be estimated with the help of  Lemma \ref{lem:gen_kato_ponce}. To verify that $J^s(\chi^{(2)}_\mu \mathcal{T}_{\mu}^{-\frac{1}{2}}) \in \mathtt{S}^{s+\frac{1}{2}}$, it suffices to show that
    \begin{align*}
        \sup_{\xi\in\R^d}  \jbrack{\xi}^{|\alpha| - (s+\frac{1}{2})} \Big| \frac{\partial^\alpha}{\partial\xi^\alpha} \Big(\jbrack{\xi}^s \chi^{(2)}_\mu \frac{1}{\sqrt{\mathcal{T}_\mu(\xi)}}\Big) \Big|  \lesssim \mu^\frac{1}{4}.
    \end{align*}
    We first write
    \begin{align*}
        \frac{1}{\sqrt{\mathcal{T}_\mu(\xi)}}&= \mu^\frac{1}{4}|\xi|^\frac{1}{2}
        \frac{1}{(\tanh(\sqrt{\mu}|\xi|))^\frac{1}{2}}.
    \end{align*}
    Then, on the support of $\chi^{(2)}$ and its derivative, we have that $\frac{1}{(\tanh(\sqrt{\mu}|\xi|))^\frac{1}{2}}$ and its derivatives are uniformly bounded, and $|\xi| \geq \sqrt{\mu}|\xi|>1$. In other words, $J^s(\chi^{(2)}_\mu \mathcal{T}_{\mu}^{-\frac{1}{2}}) \in \mathtt{S}^{s+\frac{1}{2}}$ and $\mathcal{N}^{s+\frac{1}{2}}(J^s(\chi^{(2)}_\mu \mathcal{T}_{\mu}^{-\frac{1}{2}}))\lesssim \mu^\frac{1}{4}$. Thus, we conclude by \eqref{eq:gen_kato_ponce_Linfty} that
    \begin{align}\label{eq:B3}
       \mathscr{A}_2 &\lesssim
        \mu^\frac{1}{4} \norm{\nabla f}_{L^\infty_x}\norm{g}_{H^{s-\frac{1}{2}}_x} + \mu^\frac{1}{4}\norm{\nabla f}_{H^{s-\frac{1}{2}}_x}\norm{ g}_{L^\infty_x}
        \\ \notag
        &\lesssim \norm{\nabla f}_{L^\infty_x} \norm{J^{s-1} J^\frac{1}{2}_\mu g}_{L^2_x} + \norm{J^s J^\frac{1}{2}_\mu f}_{L^2_x}\norm{ g}_{L^\infty_x}.
    \end{align}
    We complete the proof of \eqref{eq:commutator_est_tilbert_inv} by combining \eqref{eq:commutator_part_123} with \eqref{eq:B1} and \eqref{eq:B3}. 

\end{proof}

The next estimates are found in \cite[Corollary 2.6]{Paulsen2022}.
\begin{cor}
    Let $f\in \mathcal{S}(\R^d)$, $\mu \in (0,1]$, and $s\in \R$. Then
    \begin{align}\label{eq:est_on_tilbert}
        \norm{\sqrt{\mathcal{T}_\mu}(D) f}_{L^2_x} &\leq \norm{f}_{L^2_x},
        \\
        \label{eq:est_on_tilbert_J_mu}
        \norm{\sqrt{\mathcal{T}_\mu}(D)J^\frac{1}{2}_\mu f}_{L^2_x}  &\sim \norm{f}_{L^2_x},
       \\
        \label{eq:J_mu_to_f_H_s}
       \norm{J^\frac{1}{2}_\mu f}_{H^s}^2 &\sim \norm{f}_{H^s_x}^2 + \sqrt{\mu} \norm{|D|^\frac{1}{2} f}_{H^s_x}^2,
    \intertext{and}
        \label{eq:est_on_tilbert_invers}
        \norm{\mathcal{T}_{\mu}^{-\frac{1}{2}}(D) f}_{H^s_x}^2 &\sim \norm{f}_{H^s_x}^2 + \sqrt{\mu} \norm{|D|^\frac{1}{2} f}_{H^s_x}^2.
    \end{align}
\end{cor}

We also need the following standard properties of the Riesz transformation.
\begin{lem}\label{lem:riesz_transform}
    For $1<p<\infty$, it holds that
    \begin{align}\label{eq:riesz_transform_est}
        \norm{R_jf}_{L^p_x} \lesssim \norm{f}_{L^p_x}
    \end{align}
    for $j=1,2$ and
    \begin{align}\label{eq:riesz_transform_property}
        |D| &= \mathbf{R} \cdot \nabla.
    \end{align}
\end{lem}

\subsection{Energy estimates for \eqref{eq:whitham_boussinesq_d2}} \label{sec:energy_ests_d_2}
As a consequence of the curl-free condition on the initial datum in Theorem \ref{thm:wellposedness_WB}, we can obtain that $\mathrm{curl} \:\mathbf{v} = 0$. This in turn implies that the system \eqref{eq:whitham_boussinesq_d2} has the same structure as \eqref{eq:whitham_boussinesq}. Indeed, first applying the curl on the second equation in \eqref{eq:whitham_boussinesq_d2}, and then using the fundamental theorem of calculus and $\mathrm{curl} \:\mathbf{v}_0 = 0$, we obtain $\mathrm{curl} \:\mathbf{v} = 0$. In other words, we have the following relation
\begin{align}\label{eq:curl_free_cond}
    \partial_{x_2} v_1 = \partial_{x_1} v_2.
\end{align}
As mentioned in Section \ref{sec:strategy_outline}, we aim to sharpen the energy estimates presented in \cite{Paulsen2022}. We recall that the energy associated to \eqref{eq:whitham_boussinesq_d2} from Definition \ref{def:energy} is of the from
\begin{align}\label{eq:energy_func_d_2}
     E_s(\eta,\mathbf{v}) = \int_{\R^2} \Big\{ \Big( J^s \eta \Big)^2 + (1+\epsilon\eta)\Big|J^s \mathcal{T}_{\mu}^{-\frac{1}{2}}(D) \mathbf{v} \Big|^2 \Big\} \d x.
\end{align}
Let $\mathbf{U}=(\eta,\mathbf{v})^\mathrm{T}$, then the curl-free condition \eqref{eq:curl_free_cond} allows us to rewrite \eqref{eq:whitham_boussinesq_d2} as
\begin{align}\label{eq:compact_eq_d_2}
    \partial_t \mathbf{U} + \mathcal{M}(\mathbf{U},D)\mathbf{U} &= \mathbf{0},
\end{align}
where 
\begin{align}\label{eq:M_d_2}
    \mathcal{M}(\mathbf{U},D) = \begin{pmatrix}
       \epsilon \mathbf{v}\cdot \nabla & (1+\epsilon\eta)\partial_{x_1} &(1+\epsilon\eta)\partial_{x_2}\\
        \mathcal{T}_\mu(D)\partial_{x_1} & \epsilon(\mathbf{v}\cdot \nabla) & 0\\
        \mathcal{T}_\mu(D)\partial_{x_2} &  0 & \epsilon(\mathbf{v}\cdot \nabla)
    \end{pmatrix}.
\end{align}
The symmetrizer for \eqref{eq:compact_eq_d_2} is defined by
\begin{align}\label{eq:Q_d_2}
    \mathcal{Q}(\mathbf{U},D) = \begin{pmatrix}
      \mathcal{T}_\mu(D)& 0 &0\\
     0&  1+\epsilon\eta &0\\
     0& 0& 1+\epsilon\eta 
    \end{pmatrix}
\end{align}
which lets us write the energy in \eqref{eq:energy_func_d_2} as
\begin{align}\label{eq:energy_d_2}
    E_s (\mathbf{U})&= \scalarprod{J^s \mathcal{T}_{\mu}^{-\frac{1}{2}}(D)\mathbf{U}}{\mathcal{Q}(\mathbf{U},D)J^s \mathcal{T}_{\mu}^{-\frac{1}{2}}(D)\mathbf{U} }.
\end{align}
With this, we can find the following energy estimates from \eqref{eq:compact_eq_d_2}. This is summarized in the proposition below.
\begin{prop}\label{prop:energy_ests_d_2}
    Let $s >2$, $\epsilon,\mu \in(0,1]$, $r\geq s+1$, and $(\eta,\mathbf{v}) \in C([0,T];V_\mu^r(\R^2))$ be a solution to \eqref{eq:compact_eq_d_2} on a time interval $[0,T]$ for some $T>0$. Assume that there exists  a $h_0\in(0,1)$ such that
    \begin{align}\label{eq:assump_eta1_v1_d_2}
        h_0 - 1 \leq  \epsilon \eta(x,t) \  \text{ for all } (x,t)\in \R^2 \times [0,T] .
    \end{align}
    Additionally, assume that $\mathrm{curl}\:  \mathbf{v} = 0$.
    Then, the energy defined in \eqref{eq:energy_d_2} satisfies
    \begin{align}\label{eq:energy_sim_V_d_2} 
       h_0\norm{(\eta,\mathbf{v})(t)}_{V_\mu^s}^2 \lesssim E_s(\mathbf{U}(t)) \lesssim H(t) \norm{(\eta,\mathbf{v})(t)}_{V_\mu^s}^2 
    \end{align}
    and
    \begin{align}
        \label{eq:derivative_energy_d_2}
        \frac{\d}{\d t} E_s(\mathbf{U}(t)) &\lesssim  \epsilon H(t) \mathcal{P}(t)  E_s(\mathbf{U}(t))
    \end{align}
    for all $t\in (0,T)$, where
    \begin{align}\label{eq:P}
        \mathcal{P}(t) &:= \norm{\nabla \eta(t)}_{L^\infty_x} +\norm{\nabla \mathbf{v}(t)}_{L^\infty_x}  + \norm{\mathcal{T}_{\mu}^{-\frac{1}{2}}(D)\nabla \mathbf{v}(t)}_{L^\infty_x} 
        \intertext{and}
        H(t) &:= 1+ \epsilon\norm{(\eta,\mathbf{v})(t)}_{L^\infty_x \times L^\infty_x}.
    \end{align}
\end{prop}
\begin{remark}
    Compared to \cite[Proposition 3.3]{Paulsen2022}, we need sharper estimates to obtain an enhanced lifespan. For this, we need the $L^\infty$-bounds in \eqref{eq:derivative_energy_d_2}, which changes the proof of \cite[Proposition 3.3]{Paulsen2022} (see the estimate for $\mathcal{A}_{21}$). We therefore include all the details in these calculations. To simplify the calculations, we have slightly modified the energy in \cite[Equation (3.13)]{Paulsen2022}.
\end{remark}
\begin{remark}
    The estimate in \eqref{eq:derivative_energy_d_2} will be used to derive a priori estimates on the enhanced lifespan, where we will need to estimate $\norm{\mathcal{P}}_{L^2(0,T)}$ (see Proposition \ref{prop:boundedness_eta_v_d_2} below). We will therefore need to use these estimates for $s>\frac{9}{4}$.
\end{remark}
\begin{proof}
    To verify that $E_s(\mathbf{U})$ defined in \eqref{eq:energy_d_2} satisfies the equivalence \eqref{eq:energy_sim_V_d_2}. The non-cavitation condition \eqref{eq:assump_eta1_v1_d_2} and \eqref{eq:equivalent_norm_Vsmu} implies
    \begin{align*}
        E_s(\mathbf{U}) 
        &\geq  \norm{\eta}_{H^s_x}^2 + h_0  \norm{\mathcal{T}_{\mu}^{-\frac{1}{2}}(D)  \mathbf{v}}_{H^s_x}^2 
        \gtrsim h_0 \norm{(\eta,\mathbf{v})}^2_{V^s_\mu}.
    \end{align*}
    We obtain the reverse inequality by H\"{o}lder's inequality and the estimate \eqref{eq:equivalent_norm_Vsmu}. Indeed, we obtain 
    \begin{align*}
         E_s(\mathbf{U}) &\leq \norm{\eta}_{H^s_x}^2 + (1+\epsilon\norm{\eta}_{L^\infty_x} ) \norm{\mathcal{T}_{\mu}^{-\frac{1}{2}}(D) \mathbf{v}}_{H^s_x}^2 
         \lesssim H(t) \norm{(\eta,\mathbf{v})}^2_{V^s_\mu}.
    \end{align*}
    We will now prove \eqref{eq:derivative_energy_d_2} by using \eqref{eq:energy_d_2}. From the fact that $\mathcal{Q}(\mathbf{U},D)$ is symmetric and \eqref{eq:compact_eq_d_2}, we can write
    \begin{align}\label{eq:est_energy_d_2}
         \frac{1}{2} \frac{\d}{\d t} E_s(\mathbf{U}) &= - \scalarprod{J^s \mathcal{T}_{\mu}^{-\frac{1}{2}}(D) \mathcal{M}(\mathbf{U},D) \mathbf{U}}{\mathcal{Q}(\mathbf{U},D)J^s \mathcal{T}_{\mu}^{-\frac{1}{2}}(D) \mathbf{U}}_{L^2_x} \\ \notag
         &+ \frac{1}{2} \scalarprod{J^s \mathcal{T}_{\mu}^{-\frac{1}{2}}(D) \mathbf{U}}{(\partial_t\mathcal{Q}(\mathbf{U},D)) J^s \mathcal{T}_{\mu}^{-\frac{1}{2}}(D) \mathbf{U}}_{L^2_x}\\ \notag
         &=: -\mathcal{I} + \mathcal{II}.
    \end{align}
    \SkipTocEntry\subsubsection*{\underline{Control of $\mathcal{I}$}}
    With the definitions \eqref{eq:M_d_2}--\eqref{eq:Q_d_2}, the term $\mathcal{I}$ can be written as
    \begin{align*}
        \mathcal{I} &= \epsilon\scalarprod{J^s (\mathbf{v}\cdot \nabla \eta)}{J^s \eta}_{L^2_x} 
        + \scalarprod{J^s \big((1+\epsilon\eta) \nabla \cdot \mathbf{v}\big)}{J^s \eta}_{L^2_x} \\
        &+ \scalarprod{J^s \sqrt{\mathcal{T}_\mu}(D) \nabla \eta }{(1+\epsilon\eta)J^s \mathcal{T}_{\mu}^{-\frac{1}{2}}(D) \mathbf{v} }_{L^2_x} \\
        &+\epsilon \scalarprod{J^s \mathcal{T}_{\mu}^{-\frac{1}{2}}(D) \big( (\mathbf{v}\cdot \nabla) \mathbf{v}\big)}{ (1+\epsilon\eta)J^s\mathcal{T}_{\mu}^{-\frac{1}{2}}(D)\mathbf{v}}_{L^2_x} \\
        &=: \mathcal{A}_{11} +\mathcal{A}_{12} + \mathcal{A}_{21} + \mathcal{A}_{22}. 
    \end{align*}
    \SkipTocEntry\subsubsection*{Estimating $\mathcal{A}_{11}$} We decompose
    \begin{align*}
        \mathcal{A}_{11} &= \epsilon  \scalarprod{[J^s,\mathbf{v}]\cdot\nabla \eta }{J^s \eta}_{L^2_x} +  \epsilon  \scalarprod{\mathbf{v}\cdot J^s \nabla \eta }{J^s \eta}_{L^2_x}\\
        &=: \mathcal{A}_{11}^1 + \mathcal{A}_{11}^2.
    \end{align*}
    From the Cauchy-Schwarz inequality and the Kato-Ponce commutator estimate \eqref{eq:kato_ponce}, we have
    \begin{align*}
        |\mathcal{A}_{11}^1| 
        &\leq \epsilon \sum_{j=1,2}  \norm{[J^s,v_j]\partial_{x_j} \eta}_{L^2_x} \norm{\eta}_{H^s_x}\\
        &\lesssim  \epsilon \sum_{j=1,2}\Big(\norm{\nabla v_j}_{L^\infty_x} \norm{\eta }_{H^s_x} + \norm{\partial_{x_j} \eta}_{L^\infty_x} \norm{v_j }_{H^s_x}\Big) \norm{\eta}_{H^s_x} \\
        &\lesssim \epsilon \Big(\norm{\nabla \mathbf{v}}_{L^\infty_x} + \norm{\nabla \eta}_{L^\infty_x}\Big) \norm{(\eta,\mathbf{v})}_{V^s_\mu}^2.
    \end{align*}
    Next, the term $\mathcal{A}_{11}^2$ can be estimated from integration by parts and H\"{o}lder's inequality. This yields 
    \begin{align*}
        |\mathcal{A}_{11}^2| &\leq \frac{\epsilon}{2} \norm{\nabla \cdot \mathbf{v}}_{L^\infty_x}  \norm{J^s \eta }^2_{L^2_x} \leq \epsilon \norm{\nabla \mathbf{v}}_{L^\infty_x} \norm{(\eta,\mathbf{v})}_{V^s_\mu}^2.
    \end{align*}
    Combining these estimates gives us
    \begin{align}\label{eq:est_A11_d_2}
        |\mathcal{A}_{11} | \lesssim \epsilon \big( \norm{\nabla \eta}_{L^\infty_x} +\norm{\nabla \mathbf{v}}_{L^\infty_x}  \big) \norm{(\eta,\mathbf{v})}_{V^s_\mu}^2.
    \end{align}
    \SkipTocEntry\subsubsection*{Estimating $\mathcal{A}_{12} + \mathcal{A}_{21}$}
    We can write  $\mathcal{A}_{12}$ as
    \begin{align*}
    \mathcal{A}_{12} &= \epsilon
        \scalarprod{[J^s  ,\eta] \nabla \cdot  \mathbf{v} }{J^s \eta  }_{L^2_x}
        + \scalarprod{(1+\epsilon \eta) J^s \nabla \cdot  \mathbf{v} }{J^s \eta  }_{L^2_x}\\
        &=: \mathcal{A}_{12}^1 + \mathcal{A}_{12}^2.
    \end{align*}
    From the Cauchy-Schwarz inequality and the Kato-Ponce commutator estimate \eqref{eq:kato_ponce}, we obtain
    \begin{align*}
        |\mathcal{A}_{12}^1| &\leq \epsilon
        \norm{[J^s,\eta] \nabla \cdot  \mathbf{v} }_{L^2_x} \norm{J^s \eta  }_{L^2_x}\\
        &\lesssim  \epsilon \big(\norm{\nabla \mathbf{v}}_{L^\infty_x} + \norm{\nabla \eta}_{L^\infty_x}\big) \norm{(\eta,\mathbf{v})}_{V^s_\mu}^2.
    \end{align*}
    We cannot estimate $\mathcal{A}_{12}^{2}$ on its own, but it will be absorbed by $\mathcal{A}_{21}$. Next,  we observe from integration by parts  that
    \begin{align*}
        \mathcal{A}_{21} &= - \epsilon\scalarprod{J^s   \sqrt{\mathcal{T}_\mu}(D)  \eta }{\nabla\eta\cdot J^s \mathcal{T}_{\mu}^{-\frac{1}{2}}(D) \mathbf{v}}_{L^2_x} - \scalarprod{J^s\sqrt{\mathcal{T}_\mu}(D) \eta }{(1+\epsilon \eta)J^s \mathcal{T}_{\mu}^{-\frac{1}{2}}(D)  \nabla \cdot \mathbf{v}  }_{L^2_x}\\
        &=: \mathcal{A}_{21}^1 + \mathcal{A}_{21}^2.
        \notag
    \end{align*}
    From H\"{o}lder's inequality, \eqref{eq:est_on_tilbert}, and \eqref{eq:equivalent_norm_Vsmu}, we obtain
    \begin{align*}
        |\mathcal{A}_{21}^1| &\leq \epsilon \norm{\nabla \eta}_{L^\infty_x}\norm{J^s \sqrt{\mathcal{T}_\mu}(D)   \eta }_{L^2_x} \norm{J^s \mathcal{T}_{\mu}^{-\frac{1}{2}}(D)  \mathbf{v}  }_{L^2_x}\\
        &\lesssim \epsilon \norm{\nabla \eta}_{L^\infty_x} \norm{(\eta,\mathbf{v})}^2_{V^s_\mu}.
    \end{align*}
    For the term $\mathcal{A}_{21}^2$, we first need to rewrite it as
    \begin{align*}
        \mathcal{A}_{21}^2 &=\epsilon \scalarprod{J^s \sqrt{\mathcal{T}_\mu}(D)   \eta }{[J^s,\eta]   \mathcal{T}_{\mu}^{-\frac{1}{2}}(D)  \nabla\cdot \mathbf{v}}_{L^2_x}  - \scalarprod{J^s \sqrt{\mathcal{T}_\mu}(D)   \eta }{J^s( (1+\epsilon \eta)\mathcal{T}_{\mu}^{-\frac{1}{2}}(D)  \nabla\cdot \mathbf{v} ) }_{L^2_x} \\
        &= : \mathcal{A}_{21}^{2,1} +  \mathcal{A}_{21}^{2,2}.
    \end{align*}
    This approach is needed due to the fact that we cannot close the estimates with $\norm{J^s \nabla\cdot \mathbf{v}}_{L^\infty_x}$ or $\norm{\mathcal{T}_{\mu}^{-\frac{1}{2}}(D)\eta}_{H^s_x}$ on the right-hand side. We will therefore use the commutator estimate \eqref{eq:commutator_est_tilbert2} to handle $\mathcal{A}_{21}^{2,2}$. We start by treating $\mathcal{A}_{21}^{2,1}$. From the Cauchy-Schwarz inequality, \eqref{eq:est_on_tilbert}, the Kato-Ponce commutator estimate \eqref{eq:kato_ponce}, and \eqref{eq:est_on_tilbert_invers}, we have
    \begin{align*}
        |\mathcal{A}_{21}^{2,1}| &\leq \epsilon \norm{J^s \sqrt{\mathcal{T}_\mu}(D)  \eta }_{L^2_x} \norm{[J^s,\eta]   \mathcal{T}_{\mu}^{-\frac{1}{2}}(D)  \nabla \cdot \mathbf{v} }_{L^2_x} \\
        &\lesssim \epsilon \norm{\eta}_{H^s_x} \Big(
        \norm{\nabla \eta}_{L^\infty_x} \norm{J^s \mathcal{T}_{\mu}^{-\frac{1}{2}}(D)  \mathbf{v} }_{L^2_x} + \norm{J^s \eta}_{L^2_x} \norm{\mathcal{T}_{\mu}^{-\frac{1}{2}}(D)   \nabla\cdot \mathbf{v} }_{L^\infty_x} \Big) \\
        &\lesssim \epsilon \Big( \norm{\nabla \eta}_{L^\infty_x} +\norm{\mathcal{T}_{\mu}^{-\frac{1}{2}}(D) \nabla \mathbf{v} }_{L^\infty_x} \Big) \norm{(\eta,\mathbf{v})}_{V^s_\mu}^2.
    \end{align*}
    Lastly, we look at the term $\mathcal{A}_{21}^{2,2}$.   We observe that
    \begin{align*}
        \mathcal{A}_{21}^{2,2}
        &=- \epsilon\scalarprod{J^s   \eta }{[\sqrt{\mathcal{T}_\mu}(D)J^s,\eta ]\mathcal{T}_{\mu}^{-\frac{1}{2}}(D)  \nabla\cdot \mathbf{v}  }_{L^2_x} - \scalarprod{J^s  \eta }{ (1+\epsilon \eta) J^s \nabla\cdot \mathbf{v}  }_{L^2_x} \\
        &=:\mathcal{A}_{21}^{2,2,1} +  \mathcal{A}_{21}^{2,2,2}.
    \end{align*}
    The Cauchy-Schwarz inequality, the commutator estimate \eqref{eq:commutator_est_tilbert2}, and \eqref{eq:est_on_tilbert_invers} applied to  $\mathcal{A}_{21}^{2,2,1}$ yields
    \begin{align*}
        |\mathcal{A}_{21}^{2,2,1}| &\leq \epsilon \norm{J^s  \eta }_{L^2_x} \norm{[\sqrt{\mathcal{T}_\mu}(D)J^s,\eta ]\mathcal{T}_{\mu}^{-\frac{1}{2}}(D)  \nabla\cdot \mathbf{v}  }_{L^2_x}\\
        &\lesssim \epsilon \norm{\eta}_{H^s_x} \Big( \norm{\nabla \eta}_{L^\infty_x}\norm{J^s \mathcal{T}_{\mu}^{-\frac{1}{2}}(D)  \mathbf{v}}_{L^2_x} + \norm{J^s \eta}_{L^2_x}\norm{ \mathcal{T}_{\mu}^{-\frac{1}{2}}(D)  \nabla\cdot \mathbf{v}}_{L^\infty_x} \Big)\\
        &\lesssim \epsilon \Big( \norm{\nabla \eta}_{L^\infty_x} +\norm{\mathcal{T}_{\mu}^{-\frac{1}{2}}(D) \nabla\cdot \mathbf{v} }_{L^\infty_x} \Big) \norm{(\eta,\mathbf{v})}_{V^s_\mu}^2.
    \end{align*}
    We make the observation that $\mathcal{A}_{21}^{2,2,2} = - \mathcal{A}_{12}^{2}$. So we may gather the above estimates to obtain
    \begin{align}\label{eq:est_A12A21_d_2}
        |\mathcal{A}_{12} + \mathcal{A}_{21}| \lesssim \epsilon \Big( \norm{\nabla \eta}_{L^\infty_x} + \norm{\nabla \mathbf{v}}_{L^\infty_x} + \norm{\mathcal{T}_{\mu}^{-\frac{1}{2}}(D)  \nabla \mathbf{v} }_{L^\infty_x}\Big) \norm{(\eta,\mathbf{v})}_{V^s_\mu}^2.
    \end{align}

   \SkipTocEntry \subsubsection*{Estimating $\mathcal{A}_{22}$} 
    We first observe that
    \begin{align*}
        (\mathbf{v} \cdot \nabla)\mathbf{v} = (v_1\partial_{x_1} v_1 + v_2 \partial_{x_2}v_1, v_1\partial_{x_1} v_2 + v_2 \partial_{x_2}v_2)^\mathrm{T}
    \end{align*}
    which allows us to write
    \begin{align*}
        \mathcal{A}_{22} &= \epsilon \sum_{j,\ell = 1,2} \scalarprod{J^s \mathcal{T}_{\mu}^{-\frac{1}{2}}(D) (v_j \partial_{x_j} v_\ell)}{ (1+\epsilon\eta)J^s\mathcal{T}_{\mu}^{-\frac{1}{2}}(D)v_\ell}_{L^2_x}\\ \notag
        &=\epsilon \sum_{j,\ell = 1,2} \scalarprod{[J^s \mathcal{T}_{\mu}^{-\frac{1}{2}}(D) , v_j] \partial_{x_j} v_\ell}{ (1+\epsilon\eta)J^s\mathcal{T}_{\mu}^{-\frac{1}{2}}(D)v_\ell}_{L^2_x} \\ \notag
        &+ \epsilon \sum_{j,\ell = 1,2} \scalarprod{v_jJ^s \mathcal{T}_{\mu}^{-\frac{1}{2}}(D)  \partial_{x_j} v_\ell}{ (1+\epsilon\eta)J^s\mathcal{T}_{\mu}^{-\frac{1}{2}}(D)v_\ell}_{L^2_x} \\ \notag
        &=: \mathcal{A}_{22}^1 + \mathcal{A}_{22}^2.
    \end{align*}
    With the help of H\"{o}lder's inequality, the commutator estimate \eqref{eq:commutator_est_tilbert_inv},  and \eqref{eq:J_mu_to_f_H_s}--\eqref{eq:est_on_tilbert_invers}, we have
    \begin{align*}
        |\mathcal{A}_{22}^1| &\leq \epsilon \sum_{j,\ell = 1,2}   \norm{[J^s \mathcal{T}_{\mu}^{-\frac{1}{2}}(D) ,v_j] \partial_{x_j} v_\ell }_{L^2_x}(1+\epsilon \norm{\eta}_{L^\infty_x})\norm{J^s \mathcal{T}_{\mu}^{-\frac{1}{2}}(D) v_\ell }_{L^2_x}\\ \notag
        &\lesssim H(t)  \epsilon\sum_{j,\ell = 1,2}   \Big( \norm{\nabla v_j}_{L^\infty_x}
        \norm{J^\frac{1}{2}_\mu v_\ell}_{H^s_x} + \norm{\partial_{x_j} v_\ell}_{L^\infty_x}
        \norm{J^\frac{1}{2}_\mu v_j}_{H^s_x}  \Big)\norm{J^s \mathcal{T}_{\mu}^{-\frac{1}{2}}(D) v_\ell }_{L^2_x}\\ \notag
        &\lesssim \epsilon H(t) \norm{\nabla \mathbf{v}}_{L^\infty_x}\norm{(\eta, \mathbf{v})}_{V^s_\mu}^2.
    \end{align*}
    The term $\mathcal{A}_{22}^2$ is treated by integration by parts, H\"{o}lder's inequality,   and \eqref{eq:est_on_tilbert_invers}. This gives us
    \begin{align*}
        |\mathcal{A}_{22}^2| &\leq \frac{\epsilon}{2} \sum_{j,\ell = 1,2}\Big|  \scalarprod{ (\partial_{x_j}v_j)J^s \mathcal{T}_{\mu}^{-\frac{1}{2}}(D)  v_\ell}{ (1+\epsilon\eta)J^s\mathcal{T}_{\mu}^{-\frac{1}{2}}(D)v_\ell}_{L^2_x}\Big| \\
        &+ \frac{\epsilon^2 }{2} \sum_{j,\ell = 1,2}\Big| \scalarprod{ v_j J^s \mathcal{T}_{\mu}^{-\frac{1}{2}}(D) v_\ell }{(\partial_{x_j} \eta) J^s \mathcal{T}_{\mu}^{-\frac{1}{2}}(D) v_\ell  }_{L^2_x}\Big|\\
        &\lesssim  \epsilon H(t)\Big(\norm{\nabla \mathbf{v}}_{L^\infty_x}+\norm{\nabla \eta}_{L^\infty_x}\Big)\norm{(\eta, \mathbf{v})}_{V^s_\mu}^2. 
    \end{align*}
    We gather the estimates for $\mathcal{A}_{22}^1$ and $\mathcal{A}_{22}^2$ to obtain
    \begin{align}\label{eq:est_A22_d_2}
        |\mathcal{A}_{22}| &\lesssim  \epsilon H(t)\Big(\norm{\nabla \mathbf{v}}_{L^\infty_x}+ \norm{\nabla \eta}_{L^\infty_x}\Big)\norm{(\eta, \mathbf{v})}_{V^s_\mu}^2.
    \end{align}
    Combining \eqref{eq:est_A11_d_2}--\eqref{eq:est_A22_d_2} yields
    \begin{align}\label{eq:mathcal_I_d_2}
        |\mathcal{I}| \lesssim \epsilon  H(t) \Big( \norm{\nabla \eta}_{L^\infty_x} + \norm{\nabla \mathbf{v}}_{L^\infty_x} + \norm{\mathcal{T}_{\mu}^{-\frac{1}{2}}(D)\nabla \mathbf{v} }_{L^\infty_x}\Big) \norm{(\eta,\mathbf{v})}_{V^s_\mu}^2.
    \end{align}
    \SkipTocEntry\subsubsection*{\underline{Control of $\mathcal{II}$}}
    From \eqref{eq:compact_eq_d_2} and \eqref{eq:Q_d_2}, we can write
    \begin{align*}
       \mathcal{II} &=\frac{\epsilon}{2} \scalarprod{J^s\mathcal{T}_{\mu}^{-\frac{1}{2}}(D)  \mathbf{v}}{(\partial_t \eta) J^s\mathcal{T}_{\mu}^{-\frac{1}{2}}(D)  \mathbf{v}}_{L^2_x}\\
       &= -\frac{\epsilon}{2} \scalarprod{J^s\mathcal{T}_{\mu}^{-\frac{1}{2}}(D)  \mathbf{v}}{(\nabla \cdot \mathbf{v}) J^s\mathcal{T}_{\mu}^{-\frac{1}{2}}(D)  \mathbf{v}}_{L^2_x} -\frac{\epsilon^2}{2} \scalarprod{J^s\mathcal{T}_{\mu}^{-\frac{1}{2}}(D)  \mathbf{v}}{(\nabla\cdot(\eta  \mathbf{v})) J^s\mathcal{T}_{\mu}^{-\frac{1}{2}}(D)  \mathbf{v}}_{L^2_x}.
    \end{align*}
    Then from H\"{o}lder's inequality and \eqref{eq:equivalent_norm_Vsmu}, we obtain
    \begin{align}\label{eq:mathcal_II_d_2}
        |\mathcal{II}|&\lesssim \epsilon \Big(\norm{\nabla \eta}_{L^\infty_x} + \norm{\nabla \mathbf{v}}_{L^\infty_x}\Big) \norm{(\eta,\mathbf{v})}_{V^s_\mu}^2.
    \end{align}
    Combining \eqref{eq:est_energy_d_2} with \eqref{eq:mathcal_I_d_2}--\eqref{eq:mathcal_II_d_2} gives us \eqref{eq:derivative_energy_d_2}, which concludes this proof.
\end{proof}

With the results in Proposition \ref{prop:energy_ests_d_2}, an application of Gr\"{o}nwall's inequality, and approximations, we may deduce the following estimate.

\begin{cor}\label{cor:wb_gron}
     Let $s >2$, $\epsilon,\mu\in(0,1]$, and  $T>0$. Assume that there exists  a $h_0\in(0,1)$ such that
    \begin{align}\label{eq:assump_eta1_v1_d_2_part2}
        h_0 - 1 \leq  \epsilon \eta(x,t) \  \text{ for all } (x,t)\in \R^2 \times [0,T]
    \end{align}
    and that $\mathrm{curl}\:  \mathbf{v} = 0$.
    Then, for any solution $(\eta,\mathbf{v}) \in C([0,T];V_\mu^s(\R^2))$ to \eqref{eq:compact_eq_d_2} on $[0,T]$ with $(\eta(0),\mathbf{v}(0))=(\eta_0,\mathbf{v}_0)$ we have
     \begin{equation}\label{eq:gronwall_wb}
         \norm{(\eta,\mathbf{v})(t)}_{V_\mu^s}^2
         \lesssim H(t) h_0^{-1}
         \mathrm{e}^{c  \epsilon \int_0^tH(\tau) \mathcal{P}(\tau) \!\d \tau} \norm{(\eta_0,\mathbf{v}_0)}_{V_\mu^s}^2
     \end{equation}
     for $t\in [0,T]$, where 
     \begin{align*}
         \mathcal{P}(t) &= \norm{\nabla \eta(t)}_{L^\infty_x} +\norm{\nabla \mathbf{v}(t)}_{L^\infty_x}  + \norm{\mathcal{T}_{\mu}^{-\frac{1}{2}}(D)\nabla \mathbf{v}(t)}_{L^\infty_x} 
          \intertext{and}
        H(t) &= 1+ \epsilon\norm{(\eta,\mathbf{v})(t)}_{L^\infty_x \times L^\infty_x}.
     \end{align*}
\end{cor}
\begin{proof}
    The proof follows the same arguments as Corollary \ref{cor:energy1_whitham}; we therefore omit this proof.
\end{proof}


\subsection{Refined Strichartz estimate for \eqref{eq:whitham_boussinesq_d2}}\label{sec:refined_str_wb}
To obtain the enhanced lifespan in Theorem \ref{thm:wellposedness_WB}, we need to estimate $\norm{\mathcal{P}}_{L^2(0,T)}$ appearing in \eqref{eq:gronwall_wb} (see Proposition \ref{prop:boundedness_eta_v_d_2} and Lemma \ref{lem:bootstrap_wb} below). We first have to write \eqref{eq:whitham_boussinesq_d2} in its equivalent form with the linear part being diagonal. From this, we will see that the unitary group associated to the linear part of the diagonalized system is given by \eqref{eq:unitary_group}, allowing us to utilize Corollary \ref{cor:refined_strichartz_d_1}.  

We diagonalize \eqref{eq:whitham_boussinesq_d2} by defining 
\begin{align}\label{eq:def_uplus_uminus_d2}
    u^+ = \frac{1}{2}\Big( \eta - i\mathcal{T}_{\mu}^{-\frac{1}{2}}(D)\mathbf{R} \cdot \mathbf{v}\Big) \   &\text{ and } \ u^- =\frac{1}{2} \Big( \eta + i\mathcal{T}_{\mu}^{-\frac{1}{2}}(D) \mathbf{R} \cdot \mathbf{v}\Big).
    \end{align}
Moreover, $\mathbf{v}$ being curl-free yields $\nabla (\nabla \cdot \mathbf{v}) = \Delta \mathbf{v} = -|D|^2 \mathbf{v}$, so that we can reconstruct
\begin{align}\label{eq:def_eta_and_v_d_2}
    \eta = u^+ + u^- \   &\text{ and } \ \mathbf{v} = -i\sqrt{\mathcal{T}_\mu}(D)\mathbf{R}(u^+ - u^-).
\end{align}
Using \eqref{eq:whitham_boussinesq_d2}, \eqref{eq:riesz_transform_property}, and \eqref{eq:def_eta_and_v_d_2}, we obtain
\begin{align*}
    2 \sqrt{\mathcal{T}_\mu}(D) \partial_t u^\pm &= \sqrt{\mathcal{T}_\mu}(D) \partial_t \eta \mp i\mathbf{R} \cdot \partial_t \mathbf{v}\\
    &= - \sqrt{\mathcal{T}_\mu}(D)\nabla \cdot \mathbf{v} - \epsilon \sqrt{\mathcal{T}_\mu}(D)\nabla \cdot (\eta \mathbf{v})  \pm i\mathcal{T}_\mu(D) \mathbf{R} \cdot\nabla \eta \pm i\frac{\epsilon}{2}\mathbf{R} \cdot\nabla(|\mathbf{v}|^2)\\
    &=  i\mathcal{T}_\mu(D)|D| (u^+ - u^-) \pm i\mathcal{T}_\mu(D) |D| (u^+ + u^-)
    - \epsilon \sqrt{\mathcal{T}_\mu}(D)\nabla\cdot (\eta \mathbf{v})  \pm i \frac{\epsilon}{2}|D|(|\mathbf{v}|^2).
\end{align*}
We can then write
\begin{align}\label{eq:system_dig_d2}
    \begin{dcases}
       i \partial_t u^+ +  \sqrt{\mathcal{T}_\mu}(D) |D| u^+ +\frac{\epsilon}{2}\mathcal{B}^+(\eta,\mathbf{v})=0\\
       i \partial_t u^-  - \sqrt{\mathcal{T}_\mu}(D) |D| u^- + \frac{\epsilon}{2}\mathcal{B}^-(\eta,\mathbf{v})=0
    \end{dcases}
\end{align}
for 
\begin{align}\label{eq:G_nonlin}
    \mathcal{B}^\pm(\eta,\mathbf{v}) := i \nabla\cdot (\eta \mathbf{v})  \pm \frac{1}{2}\mathcal{T}_{\mu}^{-\frac{1}{2}}(D)|D|(|\mathbf{v}|^2),
\end{align}
where the initial values are transformed to $u^\pm(0) = \frac{1}{2}\Big(\eta_0 \mp i \mathcal{T}_{\mu}^{-\frac{1}{2}}(D)\mathbf{R} \cdot \mathbf{v}_0 \Big)$. We notice that 
\begin{align*}
    \eta\mathbf{v} &=-i (u^+ + u^-) \sqrt{\mathcal{T}_\mu}(D)\mathbf{R}(u^+ - u^-)\\
    \intertext{and} 
    |\mathbf{v}|^2&= |\sqrt{\mathcal{T}_\mu}(D)\mathbf{R}(u^+ - u^-)|^2,
\end{align*}
but for simplicity in later calculations, we keep the nonlinearity \eqref{eq:G_nonlin} depending on $\eta$ and $\bf v$.
\\
\indent 
We observe that  the linear part of \eqref{eq:system_dig_d2} is given by
\begin{align}\label{eq:system_dig_d2_linear}
    \begin{dcases}
       i\partial_t u^{\pm}  \pm  \sqrt{\mathcal{T}_\mu}(D) |D| u^\pm  =0,\\
       u^\pm(0) = \frac{1}{2}\Big(\eta_0 \mp i\mathcal{T}_{\mu}^{-\frac{1}{2}}(D)\mathbf{R} \cdot \mathbf{v}_0 \Big)
    \end{dcases}
\end{align}
with the solution being associated with the unitary group $S_{\mu,2} (t)$ given by
\begin{align}\label{eq:unitary_group_sol_linear_syst_d_2}
    S_{\mu,2} (\mp t) u^\pm(0) = \Big( \mathrm{e}^{\mp i(t / \sqrt{\mu}) m_2 (\sqrt{\mu}\cdot)} \widehat{u^\pm}(0) \Big)^\vee,
\end{align}
where $m_2(\xi) = |\xi|  \sqrt{\mathcal{T}_1(\xi)} = |\xi| \Big(\frac{\tanh(|\xi|)}{|\xi|}\Big)^{\frac{1}{2}}$. Since \eqref{eq:unitary_group_sol_linear_syst_d_2} coincides with \eqref{eq:unitary_group}, we may utilize the refined Strichartz estimates in Corollary \ref{cor:refined_strichartz_d_1}.

With Corollary \ref{cor:refined_strichartz_d_1}, we may now derive the next proposition, which is the key ingredient in obtaining the enhanced lifespan in Theorem \ref{thm:wellposedness_WB}.
\begin{prop}\label{prop:boundedness_eta_v_d_2}
    Let $\epsilon,\mu\in (0,1]$, $r>\frac{9}{4}$, and $T>0$. Let $(\eta,\mathbf{v}) \in C([0,T];V_\mu^{r}(\R^2))$ be a solution to \eqref{eq:compact_eq_d_2}. 
    Then
    \begin{align*}
        \mathcal{P} := \norm{\nabla \eta}_{L^\infty_x} +\norm{\nabla \mathbf{v}}_{L^\infty_x}  + \norm{\mathcal{T}_{\mu}^{-\frac{1}{2}}(D)\nabla \mathbf{v}}_{L^\infty_x} 
    \end{align*}
    satisfies
    \begin{align}
       \label{eq:est_eta_desired_Dhalf_d_2}
      \norm{\mathcal{P}}_{L^2(0,T)} &\lesssim \mu^{(-\frac{1}{6})^+ }   T^{\frac{1}{6}^+}  \norm{(\eta,\mathbf{v})}_{L^\infty_TV^{r}_\mu} +  \epsilon  \mu^{(-\frac{1}{3})^+}  T^{\frac{5}{6}^+} \norm{(\eta,\mathbf{v})}_{L^\infty_T V^{r}_\mu}^2.
    \end{align}
\end{prop}
\begin{proof}
   
    From \eqref{eq:def_eta_and_v_d_2}, we obtain
   \begin{align*}
        \norm{\nabla\eta}_{L^2_TL^\infty_x} &= \norm{ \nabla (u^+ + u^-) }_{L^2_TL^\infty_x} \\ \notag
        &\leq \norm{\nabla u^+}_{L^2_TL^\infty_x}  + \norm{ \nabla u^-}_{L^2_TL^\infty_x}\\ \notag
        &=: \sum_\pm  \norm{ \nabla u^\pm }_{L^2_TL^\infty_x}. 
   \end{align*}
   With the help of \eqref{Refined:2d} in Corollary \ref{cor:refined_strichartz_d_1} for  $u= \nabla u^\pm$ and $F =  -i\frac{\epsilon}{2}\nabla \mathcal{B}^\pm (\eta,\mathbf{v})$ (with $\mathcal{B}^\pm (\eta,\mathbf{v})$ defined in \eqref{eq:G_nonlin}), we obtain
    \begin{align}\label{eq:est2_eta_d_2}
        \norm{\nabla \eta}_{L^2_TL^\infty_x} &\lesssim \mu^{-\frac{1}{6} + \frac{5\theta}{36}} T^{\frac{1}{6} + \frac{5\theta}{18}} \sum_\pm \Big( \norm{J^{\theta} |D|u^\pm}_{L_T^\infty L_x^2}
+ \norm{\abs{D}^{\frac{3}{2}-\frac{\theta}{3}} J^{2\theta} J_\mu^{\frac{3}{4}}  u^\pm}_{L_T^\infty L_x^2}\Big)
  \\ \notag
 & + \mu^{-\frac{1}{3}+\frac{\theta}{9}} T^{\frac{1}{3}+\frac{2\theta}{9}} \sum_\pm\norm{\abs{D}^{\frac{1}{2}-\frac{\theta}{3}} J^\theta J_\mu^{\frac{3}{4}} \mathcal{B}^\pm (\eta,\mathbf{v})}_{L_T^2 L_x^2}\\ \notag
  &=:\mu^{-\frac{1}{6}+\frac{5\theta}{36}} T^{\frac{1}{6}+\frac{5\theta}{18}} I_1 +\mu^{-\frac{1}{3}+\frac{\theta}{9}} T^{\frac{1}{3}+\frac{2\theta}{9}}  I_2.
    \end{align}
    for $\theta>0$ small. Using \eqref{eq:def_uplus_uminus_d2}, \eqref{eq:riesz_transform_est}, and \eqref{eq:equivalent_norm_Vsmu} implies
    \begin{align}\label{eq:est2_eta_d_2_part_2}
        I_1 &\lesssim  \norm{J^\frac{3}{4}_\mu J^{\frac{3}{2} +\frac{5\theta}{3}} \eta}_{L^\infty_TL^2_x} + \norm{J^\frac{3}{4}_\mu J^{\frac{3}{2} +\frac{5\theta}{3}} \mathcal{T}_{\mu}^{-\frac{1}{2}}(D) \mathbf{R} \cdot  \mathbf{v} }_{L^\infty_TL^2_x}\\ \notag
        &\lesssim \norm{(\eta,\mathbf{v})}_{L^\infty_TV^{r}_\mu}
    \end{align}
    for $r>\frac{9}{4}$. To estimate $I_2$ in \eqref{eq:est2_eta_d_2}, we use \eqref{eq:est_on_tilbert_invers}, H\"{o}lder's inequality, the Kato-Ponce inequality \eqref{eq:kato_ponce_ineq}, a Sobolev embedding, and \eqref{eq:J_mu_to_f_H_s}. This yields
    \begin{align}\label{eq:est2_eta_d_2_part_3}
        I_2 &\lesssim \norm{J^\frac{3}{4}_\mu J^{\theta} |D|^{\frac{3}{2}-\frac{\theta}{3}} (\eta \mathbf{v})}_{L^2_TL^2_x} +  \norm{J^\frac{3}{4}_\mu J^\theta   \mathcal{T}_{\mu}^{-\frac{1}{2}}(D) |D|^{\frac{3}{2}-\frac{\theta}{3}}(|\mathbf{v}|^2)}_{L^2_TL^2_x} \\ \notag
        &\lesssim T^\frac{1}{2}\Big( \norm{J^{r} (\eta\mathbf{v})}_{L^\infty_TL^2_x} +  \norm{J^{r} (|\mathbf{v}|^2)}_{L^\infty_TL^2_x} +  \mu^\frac{1}{4}\norm{J^{r+\frac{1}{2}}(|\mathbf{v}|^2)}_{L^\infty_TL^2_x}\Big) \\ \notag
        &\lesssim T^\frac{1}{2} \norm{(\eta,\mathbf{v})}_{V^{r}_\mu}^2
    \end{align}
    for $r>\frac{9}{4}$. Gathering the estimates of $I_1$ and $I_2$ in \eqref{eq:est2_eta_d_2} gives us \eqref{eq:est_eta_desired_Dhalf_d_2}.
    \\
    \indent Next, from \eqref{eq:def_eta_and_v_d_2}, we have
    \begin{align*}
        \norm{\nabla \mathbf{v}}_{L^2_TL^\infty_x} +  \norm{ \mathcal{T}_{\mu}^{-\frac{1}{2}}(D) \nabla \mathbf{v}}_{L^2_TL^\infty_x} 
        &\leq \sum_\pm  \norm{ \sqrt{\mathcal{T}_\mu}(D) \nabla\mathbf{R} u^\pm }_{L^2_TL^\infty_x} + \sum_\pm  \norm{\nabla\mathbf{R} u^\pm }_{L^2_TL^\infty_x}.
    \end{align*}
    Utilizing \eqref{Refined:2d} in Corollary  \ref{cor:refined_strichartz_d_1}, \eqref{eq:est_on_tilbert}, and 
    \eqref{eq:riesz_transform_est}, we obtain the estimate in the same way as \eqref{eq:est2_eta_d_2}--\eqref{eq:est2_eta_d_2_part_3}.
\end{proof}

\subsection{Proof of Theorem \ref{thm:wellposedness_WB} in the two-dimensional case}\label{sec:proof_main_res_two_dims}
The proof relies on an energy method combined with a refined Strichartz estimate.  

\begin{proof}[Proof of Theorem \ref{thm:wellposedness_WB} for $d=2$]

We first recall a local well-posedness result for \eqref{eq:whitham_boussinesq} and \eqref{eq:whitham_boussinesq_d2}.
\begin{thm}\label{thm:paulsen}
    For $d\in \{1,2\}$ with $s>\frac{d}{2}+1$ and $\epsilon,\mu \in (0,1]$. Assume that $(\eta_0,\mathbf{v}_0)\in V^s_\mu(\R^d)$ satisfy the non-cavitation condition \eqref{eq:noncav}.  If $d=2$, we additionally assume that $\mathrm{curl}\: \mathbf{v}_0 = 0$.  Then there exists a positive time 
    \begin{equation*}
         T = T(\norm{(\eta_0,\mathbf{v}_0)}_{V^s_\mu})= \frac{1}{ c(d,s,h_0,H_0)  \norm{(\eta_0,\mathbf{v}_0)}_{V^s_\mu}} 
    \end{equation*}
   with $H_0 = 1+ \epsilon\norm{(\eta_0,\mathbf{v}_0)}_{L^\infty_x\times L^\infty_x}$    
    such that \eqref{eq:whitham_boussinesq} and \eqref{eq:whitham_boussinesq_d2} have a unique solution
    \begin{align*}
        (\eta,\mathbf{v}) \in C([0,T];V^s_\mu(\R^d) ) \cap C^1([0,T];V^{s-1}_\mu(\R^d) )
    \end{align*}
    that satisfies 
    \begin{align*}
        \sup_{t\in [0,T]} \norm{ (\eta,\mathbf{v})}_{V^s_\mu} \lesssim \norm{ (\eta_0,\mathbf{v}_0)}_{V^s_\mu}.
    \end{align*}
    Here, the implicit constant depends on $h_0$ and $H_0$.     Additionally, the flow map is continuous with respect to the initial data. Furthermore, higher regularity persists. That is, if $(\eta_0,\mathbf{v}_0) \in V^r_\mu(\R^d)$ for some $r > s$, then $(\eta,\mathbf{v})$ belongs to $C([0, T ] ; V^r_\mu(\R^d))$ (with the same $T$).
\end{thm}
\begin{remark}
    The above theorem was proved on the timescale $1/\epsilon$  under the condition 
\begin{align*}
        0<\epsilon c \norm{(\eta_0,\mathbf{v}_0)}_{V^s_\mu} \leq 1
\end{align*}
for $s>\frac{d}{2}+1$ in \cite[Theorem 1.11]{Paulsen2022}. However, one can use the same arguments as in the proof of Lemma \ref{lem:bootstrap_wb} below, where one has to substitute the refined Strichartz estimates (Proposition \ref{prop:boundedness_eta_v_d_2}) with a classical Sobolev embedding to obtain the well-posedness result for arbitrary data.
\end{remark}

It will suffice to show that the solution exists on the timescale $\epsilon^{-\frac{5}{4}} (\mu/\epsilon)^{\frac{1}{4}^-}$ to complete the proof of  Theorem \ref{thm:wellposedness_WB} as the rest follows by  Theorem \ref{thm:paulsen}. Indeed, the uniqueness, continuous dependence on the data, and persistence of higher regularity are local-in-time properties (see Remarks \ref{rem:w_rem1}--\ref{rem:w_rem3}).

   
   From Theorem \ref{thm:paulsen}, it follows that an initial datum $(\eta_0,\mathbf{v}_0)\in V^s_\mu(\R^2)$, for $s>2$, generate a solution to \eqref{eq:compact_eq_d_2}, $(\eta,\mathbf{v}) \in C([0,T^\ast);V^s_\mu(\R^2))$, where $T^\ast\geq T(\norm{(\eta_0,\mathbf{v}_0)}_{V^s_\mu})$ is the maximal time of existence. Since $T$ is a decreasing function of its argument, we may deduce classically that the solution of \eqref{eq:compact_eq_d_2} satisfies the blow-up alternative:
\begin{align}\label{eq:BA_d2}
    \text{If } \ \ T^\ast <\infty, \text{ then } \limsup_{t\nearrow T^\ast} \norm{(\eta,\mathbf{v})(t)}_{V^s_\mu} = \infty.
\end{align}

\begin{lem}\label{lem:bootstrap_wb}
    Let $s>\frac{9}{4}$ and $\epsilon,\mu\in (0,1]$. 
    Then there exists a time
    \begin{align}\label{eq:def_T_0_d_2}
        T_1 = \epsilon^{-\frac{5}{4}}  \Big(\frac{\mu}{\epsilon}\Big)^{\frac{1}{4}^-} (C_1)^{-1}\norm{(\eta_0,\mathbf{v}_0)}_{V^s_\mu}^{(-\frac{3}{2})^+}
    \end{align}
    such that $T^\ast > T_1$ and
    \begin{align*}
        \sup_{t\in [0,T_1]}\norm{(\eta,\mathbf{v})(t)}_{V^s_\mu} \leq 5 c \sqrt{\frac{H_0}{h_0}}  \norm{(\eta_0,\mathbf{v}_0)}_{V^s_\mu}
    \end{align*}
   where 
   $H_0 = 1+ \epsilon\norm{(\eta_0,\mathbf{v}_0)}_{L^\infty_x\times L^\infty_x}$, $h_0 \in (0,1)$, and  $C_1 =C_1(h_0,H_0)$ is a positive constant fixed in the proof. 
\end{lem}

\begin{proof}
    We define
    \begin{align}\label{eq:T_tilde_nu_d_2}
        \Tilde{T} := \sup \{ T \in (0,T^\ast) : \sup_{t\in [0,T]} \norm{(\eta,\mathbf{v})}_{V^s_\mu} \leq 5 c \sqrt{\frac{H_0}{h_0}} \norm{(\eta_0,\mathbf{v}_0)}_{V^s_\mu} \}
    \end{align}
    for some $c>0$.    To obtain the result, we will argue by contradiction by assuming that $\Tilde{T} \leq T_1$. 
    \\
    \indent 
    For simplicity in notation, we denote 
    \begin{align*}
        A(t) &:= \norm{(\eta,\mathbf v)(t)}_{V^s_\mu}^2,\\
        H(t) &:=  1 + \epsilon \norm{(\eta,\mathbf{v})(t)}_{L^\infty_x \times L^\infty_x},
        \intertext{and}
     h(t) &:= \inf_{x \in \R^2} \left\{ 1+ \epsilon\eta(x,t) \right\}.
\end{align*}
    We also write $A(0) := A_0$, $H(0) := H_0$, and $h(0) := h_0$, where the non-cavitation condition can then be restated as $h_0 > 0$. We fix $T_2 \in (0,\Tilde{T})$ and claim that
        \begin{equation}\label{Apriori}
       h(t) \geq \frac{h_0}{2}
        \quad \text{and} \quad
       \sup_{t\in [0,T_2]} H(t) \leq 2H_0
    \end{equation}
     for all $t \in [0,T_2]$ implying that the condition \eqref{eq:assump_eta1_v1_d_2_part2} holds.   Then, using \eqref{eq:gronwall_wb}, the Cauchy-Schwarz inequality in time, and  the claim \eqref{Apriori}  yields
    \[
    A(t)
  \leq 4c_2 \frac{H_0}{h_0} A_0
  \mathrm{e}^{c H_0 \epsilon T_2^\frac{1}{2}\norm{\mathcal{P}}_{L^2(0,T_2)}} ,
\]
where $c_2>0$ is the constant in \eqref{eq:gronwall_wb} and $\mathcal{P}(t) = \norm{\nabla \eta(t)}_{L^\infty_x} +\norm{\nabla \mathbf{v}(t)}_{L^\infty_x}  + \norm{\mathcal{T}_{\mu}^{-\frac{1}{2}}(D)\nabla \mathbf{v}(t)}_{L^\infty_x}$.
 From \eqref{eq:est_eta_desired_Dhalf_d_2} and the assumption $T_2 < \Tilde{T}\leq T_1$, it follows that, if $s > \frac{9}{4}$, then
\begin{align*}
     \epsilon  T_2^\frac{1}{2} \norm{\mathcal{P}}_{L^2(0,T_2)} &\lesssim  \epsilon
  \mu^{(-\frac{1}{6})^+} T_2^{\frac{2}{3}^+}\norm{A}_{L^\infty(0,T_2)}^{\frac{1}{2}}
  + \epsilon^2  \mu^{(-\frac{1}{3})^+} T_2^{\frac{4}{3}^+} \norm{A}_{L^\infty(0,T_2)} \\
  &\lesssim  \epsilon
  \mu^{(-\frac{1}{6})^+} T_1^{\frac{2}{3}^+} \sqrt{\frac{H_0}{h_0} A_0}
  +  \epsilon^2  \mu^{(-\frac{1}{3})^+} T_1^{\frac{4}{3}^+}  \frac{H_0}{h_0} A_0\\
  &=: \delta + \delta^2,
\end{align*}
where 
\begin{align}\label{eq:delta}
    \delta = c \epsilon \mu^{(-\frac{1}{6})^+} T_1^{\frac{2}{3}^+} \sqrt{\frac{H_0}{h_0} A_0}.
\end{align}
Gathering the above, we have
\begin{equation}\label{A-est}
    A(t) \le 4c_2\frac{H_0}{h_0} A_0 \mathrm{e}^{ \delta + \delta^2}
\end{equation}
for all $t\in [0,T_2]$. We choose $T_1 > 0$  so small that
\begin{equation}\label{delta-condition}
4 c \delta \leq  \frac{4 c H_0}{h_0} \delta < \frac{\log 2}{2},
\end{equation}
that is, 
\begin{align*}
     \frac{ 5c H_0}{h_0} c \epsilon \mu^{(-\frac{1}{6})^+} T_1^{\frac{2}{3}^+} \sqrt{ \frac{H_0}{h_0} A_0} < \frac{\log 2}{2},
\end{align*}
which is the same as \eqref{eq:def_T_0_d_2}. We remark that \eqref{delta-condition} implies $\delta < 1$, so that $\delta+\delta^2 < 2\delta$. This gives us 
\begin{align*}
      A(t) \le 8c_2\frac{H_0}{h_0} A_0 
\end{align*}
for all $t\in [0,T_2]$ implying
\begin{align*}
     \norm{(\eta,\mathbf{v})(t)}_{V^s_\mu} \leq 3\sqrt{c_2\frac{H_0}{h_0}}\norm{(\eta_0,\mathbf{v}_0)}_{V^s_\mu}.
\end{align*}
Consequently, we have from the blow-up alternative \eqref{eq:BA_d2} that $T_2< T^\ast$. 
From \eqref{eq:BA_d2}, it follows that the solution $(\eta,\mathbf{v})$ is continuous for $t\in [0,T^\ast)$. Thus, there exists a time $\tau \in [\Tilde{T},T^\ast)$ such that 
    \begin{align*}
        \norm{(\eta,\mathbf{v})(\tau)}_{V^s_\mu} \leq 4  \sqrt{c_2\frac{H_0}{h_0}}\norm{(\eta_0,\mathbf{v}_0)}_{V^s_\mu},
    \end{align*}
    which contradicts with the definition of $\Tilde{T}$ \eqref{eq:T_tilde_nu_d_2}. Hence, $T_1 < \Tilde{T}$, where $T_1$ is as seen in \eqref{eq:def_T_0_d_2}.
    \subsubsection*{\underline{Proof of the claim \eqref{Apriori}}}
    We start by verifying the last inequality of \eqref{Apriori}. From the fundamental theorem of calculus  and H\"{o}lder's inequality in time, we have that
    \begin{align*}
      H(t)  =  1+ \epsilon\norm{(\eta,\mathbf{v})(t)}_{L^\infty_x \times L^\infty_x} &\leq  H_0 +  \epsilon T_2^\frac{1}{2} \Big(
        \norm{\partial_t\eta}_{L^2_{T_2} L^\infty_x} + 
        \norm{\partial_t\mathbf{v}}_{L^2_{T_2} L^\infty_x}\Big). 
    \end{align*}
      From \eqref{eq:compact_eq_d_2}, Proposition \ref{prop:boundedness_eta_v_d_2}, and $T_2<\Tilde{T}$, we obtain
      \begin{align*}
          \epsilon T_2^\frac{1}{2} \norm{\partial_t\eta}_{L^2_{T_2} L^\infty_x} &\leq \epsilon T_2^\frac{1}{2}  \Big(\norm{\nabla \cdot \mathbf{v}}_{L^2_{T_2} L^\infty_x} + \epsilon  \norm{\nabla \eta}_{L^2_{T_2} L^\infty_x}\norm{\mathbf{v}}_{L^\infty_{T_2} L^\infty_x}
      + \epsilon   \norm{\nabla \cdot \mathbf{v}}_{L^2_{T_2} L^\infty_x}\norm{\eta}_{L^\infty_{T_2} L^\infty_x}\Big) \\
      &\leq \epsilon T_2^\frac{1}{2} \norm{H}_{L^\infty(0,T_2)}\norm{\mathcal{P}}_{L^2(0,T_2)}\\
      &\leq 2\norm{H}_{L^\infty(0,T_2)} \delta
      \end{align*}
       for $s>\frac{9}{4}$,    where $\delta$ is defined in \eqref{eq:delta}. Next, \eqref{eq:compact_eq_d_2} and 
    \begin{align*}
        (\mathbf{v} \cdot \nabla)\mathbf{v} = (v_1\partial_{x_1} v_1 + v_2 \partial_{x_2}v_1, v_1\partial_{x_1} v_2 + v_2 \partial_{x_2}v_2)^\mathrm{T}
    \end{align*}
    yields
    \begin{align*}
        \epsilon T_2^\frac{1}{2} \norm{\partial_t\mathbf{v}}_{L^2_{T_2} L^\infty_x} &\leq \norm{\mathcal{T}_\mu(D)\nabla \eta}_{L^2_{T_2} L^\infty_x} +  \epsilon \norm{(\mathbf{v}\cdot\nabla) \mathbf{v}}_{L^2_{T_2} L^\infty_x}
        \\
        &\leq  \norm{H}_{L^\infty(0,T_2)}\Big( \norm{\mathcal{T}_\mu(D)\nabla \eta}_{L^2_{T_2} L^\infty_x} +  \norm{\nabla \mathbf{v}}_{L^2_{T_2} L^\infty_x}\Big).
    \end{align*}
    The first term can be estimated by slightly modifying the proof of $\norm{\nabla \eta}_{L^2_{T_2} L^\infty_x}$ in Proposition \ref{prop:boundedness_eta_v_d_2}, where one has to utilize \eqref{eq:est_on_tilbert}. Then from Proposition \ref{prop:boundedness_eta_v_d_2}, we obtain
    \begin{align*}
       \epsilon T_2^\frac{1}{2} \norm{\partial_t\mathbf{v}}_{L^2_{T_2} L^\infty_x}  &\leq   2\norm{H}_{L^\infty(0,T_2)} \delta
    \end{align*} 
     for $s>\frac{9}{4}$.    Gathering the above calculations gives us
    \begin{align*}
        (1-4\delta) \norm{H}_{L^\infty(0,T_2)} \leq H_0,
    \end{align*}
    which requires
        \begin{align*}
         4  \delta < \frac12,
     \end{align*}
     but this is satisfied due to \eqref{delta-condition}.
     \\
    \indent 
     Lastly, we will verify that $\eta$ satisfies the non-cavitation condition. From the fundamental theorem of calculus, the non-cavitation condition on $\eta_0$ \eqref{eq:noncav}, and H\"{o}lder's inequality, we obtain
    \begin{align*}
        1+ \epsilon\eta(x,t) = 1 + \epsilon \eta_0+ \epsilon \int_0^t \partial_t\eta(x,s) \d s  \geq  h_0 - \epsilon T_2^{\frac{1}{2}} \norm{\partial_t\eta}_{L^2_{T_2} L^\infty_x}
    \end{align*}
    for all $(x,t)\in \R^2 \times [0,T_2]$. It is enough to prove that
    \begin{align*}
       \epsilon T_2^{\frac{1}{2}} \norm{\partial_t\eta}_{L^2_{T_2} L^\infty_x}  \leq \frac{h_0}{2}.
    \end{align*}
    From the first equation in \eqref{eq:whitham_boussinesq_d2}, H\"{o}lder's inequality, the second inequality in \eqref{Apriori} (verified above), Proposition \ref{prop:boundedness_eta_v_d_2}, and the assumption $T_2<\Tilde{T}\leq T_1$, we obtain
    \begin{align*}
        \epsilon T_2^{\frac{1}{2}}\norm{\partial_t\eta}_{L^2_{T_2} L^\infty_x} &\leq \norm{\nabla \cdot \mathbf{v}}_{L^2_{T_2} L^\infty_x} + \epsilon \norm{\nabla \eta}_{L^2_{T_2} L^\infty_x}\norm{\mathbf{v}}_{L^\infty_{T_2} L^\infty_x}
      + \epsilon \norm{\nabla \cdot \mathbf{v}}_{L^2_{T_2} L^\infty_x}\norm{\eta}_{L^\infty_{T_2} L^\infty_x} \\ \notag
        &\leq 2H_0  \epsilon T_2^{\frac{1}{2}}\norm{\mathcal{P}}_{L^2(0,T_2)} \\ \notag
        &\leq 4cH_0 \delta
    \end{align*}
     for $s>\frac{9}{4}$. So, we need 
     \begin{align*}
         4c \frac{H_0}{h_0} \delta < \frac12
     \end{align*}
     to conclude, which holds by \eqref{delta-condition}.
     
    \renewcommand\qedsymbol{$\overset{\text{Lemma \ref{lem:bootstrap_wb}}}{\ensuremath{\Box}}$}
\end{proof}

\subsubsection*{Conclusion.}   
It remains to comment on the uniqueness, continuous dependence of the initial data, and the persistence of higher regularity. All these properties are known for some positive time by the local well-posedness result, Theorem \ref{thm:paulsen}, and they therefore extend to the larger time interval $[0,T_1]$ by standard arguments (see Remarks \ref{rem:w_rem1}--\ref{rem:w_rem3} for more details). This concludes the proof of the theorem for $d=2$.

\end{proof}

\section*{Acknowledgements}
\noindent
NST was supported by the VISTA program, the Norwegian Academy of Science and Letters, and Equinor. AT acknowledges the support of FDCRGP 2025-2027 (Ref. 040225FD4730),  Nazarbayev University: Probabilistic methods in the study of nonlinear dispersive and wave equations.  DP
gratefully acknowledges Jean-Claude Saut for raising, several years ago, the question of whether
dispersion could enhance the existence time of certain versions of Boussinesq-type systems in the
long wave regime.

\bibliographystyle{acm}

\bibliography{dispersive}

\end{document}